\documentclass[11pt,a4paper,oneside]{amsart}
\usepackage{comment}
\usepackage{amsmath,amssymb}
\newtheorem{theorem}{Theorem}[section]
\newtheorem{lemma}[theorem]{Lemma}

\newtheorem{proposition}[theorem]{Proposition}
\newtheorem{definition}[theorem]{Definition}

\numberwithin{equation}{section}

\def\XXint#1#2#3{{\setbox0=\hbox{$#1{#2#3}{\int}$}
 \vcenter{\hbox{$#2#3$}}\kern-.5\wd0}}

\title[Logarithmic double phase embeddings]{Logarithmic double phase embeddings with variable exponents: Necessary and Sufficient Conditions}
\author{Ankur Pandey and  Nijjwal Karak }
\thanks {The first author gratefully acknowledges the Council of Scientific and Industrial Research (CSIR) for awarding a Junior Research Fellowship (File Number: 09/1026(17526)/2024-EMR-I).} 

\begin{document}
\begin{abstract}
 In this paper, we study the necessary and sufficient conditions in the domain for Sobolev-type embedding of the space $W^{1,\Phi(\cdot,\cdot)}(\Omega)$ where $\Phi(x,t):=t^{p(x)}+ a(x) t^{q(x)}\log^{r(x)}(e+t)$ with $1\leq p(x)\leq q(x).$ We have established the sub-critical embedding in bounded John domains under some regularity assumptions on the exponents $p,$ $q,$ $r$, and $a$. Conversely, we have proved that if the embedding holds in any domain $\Omega$ in $\mathbb{R}^n,$ then $\Omega$ must satisfy the log-measure density condition. 
 \end{abstract}
\maketitle

\section{Introduction}
We assume throughout the paper that $\Omega$ is an open subset of $\mathbb{R}^n$ and we consider the logarithmic variable exponent double phase function $\Phi(x,t)= t^{p(x)}+a(x)t^{q(x)}\log^{r(x)}(e+t)$ where $a(x)\geq 0$ and the variable exponents p, q and r are continuous functions defined on $\Omega$ or $\mathbb R^n$, satisfying 
\begin{enumerate}
    \item [(p1)] $1\leq p^-:= inf_{x\in\Omega} p(x)\leq sup_{x\in\Omega} p(x)=:p^+<\infty$,
    \item [(q1)] $1\leq q^-:= inf_{x\in\Omega} q(x)\leq sup_{x\in\Omega} q(x)=:q^+<\infty$,
    
    \item[(r1)] $-\infty<r^-:=inf_{x\in\Omega} r(x)\leq sup_{x\in\Omega} r(x)=:r^+<\infty$.
\end{enumerate}  

The following three conditions on $p, q,$ and $r$ will also be used; the first two in the literature are known as log-H\"older continuous, and the third one is known as log-log-H\"older continuous:
\begin{enumerate}
\item[(p2)] $|p(x)-p(y)|\leq \frac{C}{\log(e+1/|x-y|)}$ whenever $x\in\Omega$ and $y\in\Omega$
\item[(q2)] $|q(x)-q(y)|\leq \frac{C}{\log(e+1/|x-y|)}$ whenever $x\in\Omega$ and $y\in\Omega$
\item[(r2)] $|r(x)-r(y)|\leq \frac{C}{\log(e+\log(e+1/|x-y|))}$ whenever $x\in\Omega$ and $y\in\Omega$.
\end{enumerate}

For the variable exponent Sobolev space $W^{1,p(\cdot)}(\Omega),$ the continuous Sobolev-type embedding $W^{1,p(\cdot)}\hookrightarrow L^{p^*(\cdot)}(\Omega),$ when $p^+<n,$ have been established in \cite{Die04} for bounded domains with locally Lipschitz boundary, with the condition $(p2)$ on the exponent $p.$ Other Sobolev-type embeddings were considered later in \cite{ER00, ER02, Fan10, HH08}. Embeddings of Musielak-Orlicz-Sobolev spaces have been considered by several authors, \cite{CH18, Fan12, HH17, HH19, MOS18, MOS23}. Recently, Cianchi and Diening \cite{CD24} have found optimal embeddings for general Musielak-Orlicz-Sobolev spaces. The Sobolev-type embeddings for the variable exponent double phase function $\Phi_1(x,t)=t^{p(x)}+a(x) t^{q(x)}$ have been established in \cite{HW}, \cite{BGHW22}, and for two variable functions $\Phi_2(x, t)=t^{q(x)}{(\log(e+t))}^{r(x)}$ in \cite{HMOS10}. In \cite{ABW25}, the authors have considered the logarithmic variable exponent double phase function $\Phi_3(x,t)=t^{p(x)}+a(x) t^{q(x)}\log(e+t)$ and established several properties of the space $W^{1,\Phi(\cdot,\cdot)}.$ Here, we focus on the Sobolev-type embedding of $W^{1,\Phi(\cdot,\cdot)}$ for a more general class of functions $\Phi(x,t):=t^{p(x)}+a(x) t^{q(x)}{(\log(e+t))}^{r(x)}$. 
\begin{definition}
    A bounded domain $\Omega \subset \mathbb R^n$ is called a $\delta$-John domain, $\delta>0,$ if there exists $x_0\in \Omega$ such that each point in $\Omega$ can be joined to $x_0$ by a rectifiable path $\gamma$ parameterized by its arc length such that \begin{equation*}
        B(\gamma(t), \frac{t}{\delta}) \subset \Omega
    \end{equation*}
    for all $t\in [0, l(\gamma)],$ where $l(\gamma)\leq \delta$ is the length of $\gamma.$ 
\end{definition}

\noindent First, we establish the embedding for $W^{1,\Phi(\cdot,\cdot)} (\Omega)$ for a bounded John domain $\Omega$.
\begin{theorem}\label{main theorem 1}
    Let $\Omega$ be a bounded John domain. Suppose that $$\Phi(x,t):=t^{p(x)} + a(x) t^{q(x)}(\log(e+t))^{r(x)}$$ and $$\Psi(x,t):=t^{p^*(x)} + (a(x)^{\frac{1}{q(x)}}t)^{q^*(x)}\left(\log(e+ a(x)^{\frac{1}{n}}t)\right)^{\frac{r(x)q^*(x)}{q(x)}},$$ where $1/p^*(x) =1/p(x)-1/n$ and $1/q^*(x) =1/q(x)-1/n$. 
     Assume that one of the following holds:
     \begin{itemize}
     \item[(i)] The functions $p, q:\bar\Omega\rightarrow [1, \infty)$ and $a:\bar\Omega\rightarrow [0, \infty)$ are H\"older continuous functions with $p(x)\leq q(x)$ for all $x\in \bar \Omega$, and
      \begin{equation*}
       \big(\frac{q}{p} \big)^{+}  < 1+ \frac{\gamma}{n},
      \end{equation*}
          where $\gamma$ is the H\"older exponent of $a(x).$ 
      Also, the function $r:\bar\Omega\rightarrow [0, \infty)$ is log-log-H\"older continuous with $q^++ r^+<n$.
      \item[(ii)] The functions $p:\bar\Omega\rightarrow [1, \infty),$ $q:\bar\Omega\rightarrow [1, \infty)$ are log-H\"older continuous with $p(x)\leq q(x)$ for all $x\in \bar\Omega$, and $a:\bar\Omega \rightarrow [0, \infty)$ is $q^+$-H\"older continuous and $r:\bar\Omega\rightarrow [0, \infty)$ is a measurable function with $q^++ r^+<n.$
    \end{itemize}  
      Then we have $W^{1,\Phi(\cdot,\cdot)}(\Omega) \hookrightarrow L^{\Psi(\cdot,\cdot)}(\Omega) .$ 
\end{theorem}

For the necessary part, it was shown in \cite{GKP21, GKP23} that $\Omega$ must satisfy the measure density condition to have embedding $W^{1,p(\cdot)}(\Omega)\hookrightarrow L^{p^*}(\Omega),$ if $p$ is log-H\"older continuous. Note that $\Omega$ satisfies the measure density condition if there exists a constant $c>0$ such that for every $x$ in $\bar{\Omega}$ and each $R$ in $]0,1/2],$ one has $| B_R(x)\cap \Omega|\geq cR^n.$ This condition first appeared as a necessary condition for the Sobolev embedding in \cite{HKT08} and later appeared as the same for other Sobolev-type embeddings as well, \cite{AGH20, AYY22, AYY23, AYY24, Gor17, Kar19, Kar20, Kor21}. Recently, a weaker version of the measure density condition, namely log-measure density condition (see Definition \ref{log-s}), has appeared in \cite{HK22} as a necessary condition for certain Orlicz-Sobolev embedding and also in \cite{GKP23} as a necessary condition of the Sobolev-type embedding of $W^{1,p(\cdot)}(\Omega)$ if $p$ is log-log-H\"older continuous in $\Omega.$\\

\noindent Here we prove that if the embedding holds for $W^{1,\Phi(\cdot, \cdot)}$, where $\Phi(x,t):=t^{p(x)}+ a(x) t^{q(x)}(\log(e+t))^{r(x)}$ then $\Omega$ satisfies the measure density or log-measure density conditions depending on the exponent $r(\cdot).$
\begin{theorem}\label{main theorem 3}
Let $\Omega \subset \mathbb R^n$. Let
$q:\bar\Omega\rightarrow [1, \infty) $ be log-H\"older continuous with $q^+<n$ and $0<a_0\leq a(x)\leq a_1<\infty.$ Suppose that 
        $W^{1,\Phi(\cdot,\cdot)}(\Omega) \hookrightarrow L^{\Psi(\cdot,\cdot)}(\Omega),$
    where \begin{center}
        $\Phi(x,t):=t^{p(x)}+ a(x) t^{q(x)}(\log(e+t))^{r(x)},$ 
        \end{center}
        and
    \begin{equation*}
    \Psi(x,t):=t^{p^*(x)} + (a(x)^{\frac{1}{q(x)}}t)^{q^*(x)}\left(\log(e+ a(x)^{\frac{1}{n}}t)\right)^{\frac{r(x)q^*(x)}{q(x)}},
    \end{equation*}
    where $1/p^*(x) =1/p(x)-1/n,$ and $1/q^*(x) =1/q(x)-1/n.$ Then $\Omega$ satisfies the log-measure density condition when $r(x) \geq 0 $ for all $x \in \Omega$, or $r(x) \leq  0 $ for all $x \in \Omega,$ or $r(x) < 0$ for some $x\in \Omega$ and $r(x)\geq 0$ for some $x\in \Omega$. 
\end{theorem}
The following theorem demonstrates that the lower bound assumption on $a(x)$ is no longer necessary in the special case when $q(x)=p(x).$
\begin{theorem}\label{main theorem 2}
    Let $\Omega \subset \mathbb R^n$. Let $p:\bar\Omega\rightarrow [1, \infty) $ be log-H\"older continuous with $p^+<n$ and $0\leq a(\cdot) \in L^{\infty}(\Omega). $ Suppose that 
        $W^{1,\Phi(\cdot,\cdot)}(\Omega) \hookrightarrow L^{\Psi(\cdot,\cdot)}(\Omega),$
    where $\Phi(x,t):=t^{p(x)}+ a(x) t^{p(x)}(\log(e+t))^{r(x)}$ and
    \begin{equation*}
    \Psi(x,t):=t^{p^*(x)} + (a(x)^{\frac{1}{p(x)}}t)^{p^*(x)}\left(\log(e+ a(x)^{\frac{1}{n}}t)\right)^{\frac{r(x)p^*(x)}{p(x)}},
    \end{equation*}
    where $1/p^*(x) =1/p(x)-1/n$. Then 
    \begin{itemize}
        \item[(i)] $\Omega$ satisfies the log-measure density condition when $r(x) \geq 0 $ for all $x \in \Omega$, or $r(x) < 0$ for some $x\in \Omega$ and $r(x)\geq 0$ for some $x\in \Omega$.
        \item[(ii)] $\Omega$ satisfies the measure density condition when $r(x) \leq  0 $ for all $x \in \Omega.$
        \end{itemize}
\end{theorem}

\section{Notations ans Preliminaries} 
\begin{definition}
 A function $f:(0,\infty) \rightarrow \mathbb{R}$ is called almost increasing if there exists a constant $c_1\geq 1$ such that $f(s) \leq c_1f(t)$ for all $0< s< t$. Similarly, it is called almost decreasing if there exists a constant $c_2\geq 1$ such that $c_2f(s) \geq f(t)$ for all $0< s< t .$ 
 \end{definition}
 \begin{definition}
Let $f:\Omega\times (0,\infty)\rightarrow\mathbb{R}$ and $p,q>0$.We say that f satisfies
\begin{enumerate}
\item[(i)]  $(Inc)_{p},$      if $\frac{f(x, t)}{t^p}$ is increasing in the variable $t$; 
\item[(ii)] $(aInc)_{p},$     if $\frac{f(x, t)}{t^p}$ is almost increasing in the variable $t$;
\item[(iii)]$(Dec)_{q},$      if $\frac{f(x, t)}{t^q}$ is decreasing in the variable $t$;
\item[(iv)] $(aDec)_{q},$     if $\frac{f(x, t)}{t^q}$ is almost decreasing in the variable $t$.
\end{enumerate}
\end{definition}
A weaker version of the measure density condition (defined in Section 1) is defined as follows.
\begin{definition} \label{log-s} A set $\Omega\subset\mathbb{R}^n$ satisfies the log $s$-measure density condition if there exist two constants \( c>0 \) and \( \alpha>0 \) such that 
 for every $x$ in $\bar{\Omega}$ and each $R\in (0,1/2]$ we have
\[
c R^s (\log (\frac{1}{R}))^{-\alpha} \leq | B_R(x)\cap \Omega|.
\]
If $s=n$, we say that \( \Omega \) satisfies the log-measure density condition.
\end{definition}
\begin{definition}
    We say that the exponent $p(\cdot)$ satisfies Nekvinda's decay condition, if  there exist $c\in(0,1)$ and $p_{\infty}\in[1,\infty]$ such that 
   \begin{equation*}
       \int_{\{p(x)\not=p_{\infty}\}} c^{\frac{1}{|\frac{1}{p_{\infty}}-\frac{1}{p(x)}|}}dx < \infty.
   \end{equation*}
\end{definition}
\noindent Let us recall the definition of generalized weak $\Phi$-functions and related standard notation.
\begin{definition}
   Let $(\Omega,\Sigma,\mu)$ be a complete $\sigma$-finite measurable space. A function $\Phi: \Omega\times [0,\infty)\rightarrow [0,\infty]$ is said to be a generalized $\Phi$-prefunction if $x\rightarrow \Phi(x,|f(x)|)$ is measurable for every measurable $f$, $\Phi$ is increasing in the second variable and satisfies $\Phi(x, 0)=0$, $\lim_{t\rightarrow0^+}\Phi(x, t)=0 $, and  $\lim_{t\rightarrow+\infty}\Phi(x, t)=+\infty$ for a.e. $x\in \Omega.$ We say that $\Phi$ is a generalized weak $\Phi$-function if it satisfies $(aInc)_1$ on $\Omega\times [0,\infty).$ 
We denote the set of generalized weak $\Phi$-functions by $\Phi_w(\Omega,\mu)$.
\end{definition}
\begin{lemma}
If $r(x)\geq 0,$ then the function $\Phi(x,t)=t^{p(x)}+a(x)t^{q(x)}(\log(e+t))^{r(x)}$ satisfies $(aInc)_1$, and hence we have $\Phi\in \Phi_w(\Omega).$
\end{lemma}

\begin{proof}
For $x\in\Omega$, define
\[
F_x(t)=\frac{\Phi(x,t)}{t}
=t^{p(x)-1}+ a(x)t^{q(x)-1}(\log(e+t))^{r(x)},\qquad t>0.
\]
A simple computation gives
\[
t F_x'(t)
=(p(x)-1) t^{p(x)-1}+ a(x)t^{q(x)-1} (\log(e+t))^{r(x)}
\{q(x) -1 + r(x)\frac{t}{(e+t)\log(e+t)}\}.
\]
Since $0<\frac{t}{(e+t)\log(e+t)}<1,$ $1\leq p(x)\leq q(x)$ and $r(x)\geq 0,$ we have $F_x'(t)\geq 0$ and hence $F_x(t)$ is increasing. Therefore $\Phi$ satisfies $(aInc)_1.$
\end{proof}
\noindent The generalized weak $\Phi$-functions enjoy several nice properties under the following additional assumptions.
\begin{definition}
     A function $\Phi\in\Phi_w(\Omega,\mu)$ satisfies the condition\\
     $(A0)$ if there is a constant $\beta\in(0,1]$ such that
   \begin{center} 
     $\beta\leq\Phi^{-1}(x,1)\leq{\frac{1}{\beta}}$  
     \end{center}
     hold for $\mu $-a.e. $x\in \Omega$.\\
$(A1)$ if there is a constant $\beta\in(0,1)$ such that 
     \begin{center}
      $\beta\Phi^{-1}(x,t)\leq\Phi^{-1}(y,t)$   
     \end{center}
     holds for every $t\in[1,\frac{1}{|B|}]$, $\mu$-a.e. $x,y\in B\cap\Omega$ and every ball $B $ with $|B|\leq 1$.\\
     $(A2)$ if for every $s>0$ there exist $\beta\in(0,1]$ and $h\in L^1(\Omega)\cap L^{\infty}(\Omega)$ such that \begin{center}
      $\beta\Phi^{-1}(x,t)\leq\Phi^{-1}(y,t)$   
      \end{center}
      holds for $\mu$-a.e. $x,y\in \Omega$ and every $t\in[h(x)+h(y),s]$.
\end{definition}
\begin{definition}
    Let $\Omega\subset \mathbb R^n$ and $\Phi:\Omega\times[0,\infty)\rightarrow[0,\infty]$  be a generalized weak $\Phi-$function, we say that it satisfies the condition\\
    $(A0)'$ if there is a constant $\beta\in(0,1]$ such that $\Phi(x, \beta)\leq 1\leq \Phi(x, \beta^{-1})$ hold for a.e. $x\in \Omega;$\\
    $(A1)'$ if there is a constant $\beta\in(0,1)$ such that $\Phi(x, \beta t)\leq\Phi(y, t) $ holds for every $t\geq 0$ with $\Phi(y,t)\in [1, \frac{1}{|B|}]$, and a.e. $x,y\in B\cap\Omega$ and every ball $B$ with $|B|\leq 1;$\\    
    $(A2)'$ if there exist $s>0, \beta\in (0,1]$, $h\in L^1(\Omega)\cap L^{\infty}(\Omega)$ and a weak $\Phi$-function $\phi_{\infty}$ such that the inequalities $\Phi(x, \beta t) \leq \phi_{\infty}(t)+ h(x) $ and $\phi_{\infty}(\beta t) \leq \Phi(x, t) +h(x)$  hold true for a.e. $x\in \Omega$ and all $t\geq 0$ with $\phi_{\infty}(t)\leq s$ and $\Phi(x, t)\leq s,$ respectively.
\end{definition}

For the following results, see Corollary 3.7.4, Corollary 4.1.6, and Lemma 4.2.7 in the book \cite{HH19}.
\begin{lemma}\label{equivalent lemma}
  For a generalized weak $\Phi-$function $\Phi$ the following statements are true.\\
   $(i)$ $\Phi$ satisfies the condition $(A0)$ if and only if it satisfies the condition $(A0)'$.\\
   $(ii)$ If $\Phi$ satisfies the condition $(A0)$, then the conditions  $(A1)$ and $(A1)'$ are equivalent.\\
   $(iii)$ $\Phi$ satisfies the condition $(A2)$ if and only if it satisfies the condition $(A2)'$. 
\end{lemma}
Let us recall the definition of a generalized Orlicz space or Musielak-Orlicz (M-O) space.

\begin{definition}\label{definition 2.10}
Let $\Phi\in\Phi_w(\Omega,\mu)$. F $f\in L^0(\Omega,\mu)$, the function
\begin{center}
  $\rho_\Phi(f)  := \int_\Omega\Phi(x,|f(x)|)d\mu(x)$
\end{center}	
is called a modular. The set 
\begin{center}
$ L^{\Phi(\cdot,\cdot)}(\Omega)$ := $\{f\in L^0(\Omega,\mu) :\rho_\Phi(\lambda f)< \infty  $ for some $\lambda > 0  \}$  
 \end{center}
 is called a generalized Orlicz space or Musielak-Orlicz space. Note that $L^{\Phi(\cdot,\cdot)}(\Omega)$ is a quasi-Banach space when equipped with the quasi-norm 
 \begin{center}
     $||u||_{L^{\Phi(\cdot,\cdot)}(\Omega)}= \inf\{ \lambda>0 : \rho_\Phi(\frac{u}{\lambda}) \leq 1\} .$
 \end{center}
 Also, if $\Phi$ is a generalized convex $\Phi$-function, then it becomes a Banach space. Furthermore, if $\Phi$ satisfies $(aDec)$ and $\mu$ is separable, then $ L^{\Phi(\cdot,\cdot)}(\Omega)$ is separable; and if $\Phi$ satisfies $(aDec)$ and $(aInc)$, then it is reflexive.
 \end{definition}

The next three results have been borrowed from the book by Harjulehto-H\"{a}st\"{o} \cite{HH19} (see Lemma 3.2.9, Theorem 3.2.6, and Lemma 3.2.11 in the book).

\begin{proposition}\label{proposition 2.12}
Suppose that $\Phi : \Omega \times [0,+\infty) \to [0,+\infty]$ is generalized weak $\Phi$-function which satisfies \((\text{aInc})_p\) and \((\text{aDec})_q\) for some $p,q$ with $1 \leq p \leq q < \infty$. Then for all measurable function $u$ on $\Omega,$ one has
\[
C^{-1} \min \left\{ \|u\|^p_{L^{\Phi(\cdot, \cdot)}(\Omega)}, \, \|u\|^q_{L^{\Phi(\cdot, \cdot)}(\Omega)} \right\}
   \leq \varrho_{\Phi}(u)
   \leq C \max \left\{ \|u\|^p_{L^{\Phi(\cdot, \cdot)}(\Omega)}, \, \|u\|^q_{L^{\Phi(\cdot, \cdot)}(\Omega)} \right\},
\]
where the constant $C$ is the maximum of the constants of \((\text{aInc})_p\) and \((\text{aDec})_q\).
\end{proposition}

\begin{proposition}\label{proposition 2.13}
Let $\Phi, \Psi$ be any two generalized weak $\Phi$-functions and let $\mu$ be atom less. Then $L^{\Phi(\cdot, \cdot)}(\Omega) \hookrightarrow L^{\Psi(\cdot, \cdot)}(\Omega)$ holds if and only if there exist $K > 0$ and $h \in L^1(\Omega)$ with $\|h\|_1 \leq 1$ such that for all $t \geq 0$ and for almost every\ $x \in \Omega,$ one has
\[
\Psi\!\left(x, \frac{t}{K}\right) \leq \Phi(x,t) + h(x).
\]
\end{proposition}

\begin{proposition}\label{proposition 2.14}
Let $\Phi$ be a generalized weak $\Phi$-function. Then for all $u \in L^{\Phi(\cdot, \cdot)}(\Omega)$ and all $v \in L^{\Phi^*(\cdot, \cdot)}(\Omega),$ one has
\[
\int_{\Omega} |u| \,| v| \, d\mu(x) \leq 2 \, \|u\|_{L^{\Phi(\cdot, \cdot)}(\Omega)} \, \|v\|_{L^{\Phi^*(\cdot, \cdot)}(\Omega)}.
\]
\end{proposition}
The following two results can be found as Theorems 6.1.4 and 6.1.9 in the book \cite{HH19}.

\begin{proposition}\label{proposition 2.15}
Let $\Phi : \Omega \times [0,+\infty) \to [0,+\infty]$ be a generalized weak $\Phi$-function such that $L^\varPhi(\Omega) \subseteq L^1_{\text{loc}}(\Omega)$ and $k \geq 1$. Then the set
\[
W^{k,\Phi(\cdot, \cdot)}(\Omega) = \{ u \in L^{\Phi(\cdot, \cdot)}(\Omega) : \partial_\alpha u \in L^{\Phi(\cdot, \cdot)}(\Omega) \text{ for all } |\alpha| \leq k \},
\]
where we consider the modular
\[
\varrho_{W^{k,\Phi(\cdot,\cdot)} (\Omega)}(u) = \sum_{0\leq |\alpha|\leq k} \varrho_\Phi(\partial_\alpha u)
\]
and the associated Luxemburg quasi-norm
\[
\|u\|_{W^{k,\Phi(\cdot,\cdot)} (\Omega)} = \inf \left\{ \lambda > 0 : \varrho_{W^{k,\Phi(\cdot,\cdot)} (\Omega)}\!\left(\frac{u}{\lambda}\right) \leq 1 \right\}
\]
is a quasi-Banach space. Furthermore, if $\Phi$ is a generalized convex $\Phi$-function, the space $W^{k,\Phi(\cdot, \cdot)}(\Omega)$ is a  Banach space. If $\Phi$ satisfies $(aDec)$, then it is separable; and if $\Phi$ satisfies $(aDec)$ and $(aInc)$, then it is reflexive.
\end{proposition}

\begin{proposition}\label{proposition 2.16}
Suppose that $\Phi$ is a generalized weak $\Phi$-function which satisfies \((\text{aInc})_p\) and \((\text{aDec})_q\) for some $p,q$ with $1 \leq p \leq q < \infty$. Then for all $u \in W^{k,\Phi(\cdot,\cdot)}(\Omega),$ one has
\[
C^{-1} \min \left\{ \|u\|^p_{W^{k,\Phi(\cdot,\cdot)} (\Omega)}, \, \|u\|^q_{W^{k,\Phi(\cdot,\cdot)} (\Omega)} \right\}
   \leq \varrho_{W^{k,\Phi(\cdot,\cdot)} (\Omega)}(u)
   \leq C \max \left\{ \|u\|^p_{W^{k,\Phi(\cdot,\cdot)} (\Omega)}, \, \|u\|^q_{W^{k,\Phi(\cdot,\cdot)} (\Omega)} \right\},
\]
where the constant $C$ is the maximum of the constants of \((\text{aInc})_p\) and \((\text{aDec})_q\).
\end{proposition}
Here we prove the $(Inc)$ and  $(Dec)$ properties for our function $\Phi(x,t):=t^{p(x)}+a(x)t^{q(x)}(\log(e+t))^{r(x)}.$
\begin{lemma}\label{main lemma for increasing}
   Let $\Omega \subset \mathbb R^n, n\geq 2,$ be a domain. Assume that $ p, q, r\in C(\bar\Omega),$ with $1\leq p(x)\leq q(x)$ for all $x\in\bar\Omega$, $r^+ \geq 0$ and $a \in L^1(\Omega)$. Then $\Phi(x,t)=t^{p(x)}+a(x)t^{q(x)}(\log(e+t))^{r(x)}$ satisfies $(Inc)_{p^-}$ and $(Dec)_Q$, where $Q= q^++ r^+\geq 1.$
\end{lemma}
\begin{proof}
    
   First, note that 
    \begin{equation*}
        \frac{\Phi(x, t)}{t^{p^-}}= t^{p(x)-p^-}+a(x)t^{q(x)-p^-}{(\log(e+t))}^{r(x)}
    \end{equation*}
    is an increasing function. So $\Phi(x, t)$ satisfies $(Inc)_{p^-}.$ \\
    
    The condition $(Dec)_Q$ follows, since we have, for $0\leq s\leq t,$ 
\begin{eqnarray*}
    \frac{\Phi(x,t)}{t^{Q}} &= &t^{p(x)-Q}+a(x)t^{q(x)-Q}{(\log(e+t))}^{r(x)}\\
    &\leq& s^{p(x)-Q}+a(x)s^{q(x)-Q}{(\log(e+s))}^{r(x)}\\
    &=&\frac{\Phi(x,s)}{s^Q}. 
\end{eqnarray*} \qedhere
\end{proof}
As a consequence of the previous results, we obtain some basic properties of the space $L^{\Phi(\cdot, \cdot)}(\Omega)$, where $\Phi(x,t):=t^{p(x)}+a(x)t^{q(x)}(\log(e+t))^{r(x)}.$ Here and hereafter, we use the notation $C_{+}(\bar \Omega):= \{p\in C(\bar \Omega) : p^->1\}.$
\begin{proposition}\label{proposition 2.18}
 Let  $\Omega$  be a domain in $\mathbb R^n,$ $ n\geq 2.$ Assume that $ p, q\in C_+(\bar\Omega)$ with $ p(x)\leq q(x)$ for all $x\in\bar\Omega $, $r^+ \geq 0$ and $a \in L^1(\Omega)$. Then $L^{\Phi(\cdot, \cdot)}(\Omega),$ where $\Phi(x,t)=t^{p(x)}+a(x)t^{q(x)}(\log(e+t))^{r(x)},$ is a separable, reflexive Banach space and we also have the following results:
\begin{itemize}
    \item[(i)] $\|u\|_{L^{\Phi(\cdot, \cdot)}(\Omega)} = \lambda \quad \text{if and only if } \varrho_{\Phi}\!\left(\tfrac{u}{\lambda}\right) = 1 \ \text{for } u \neq 0 \text{ and } \lambda > 0,$
    \item[(ii)] $\|u\|_{L^{\Phi(\cdot, \cdot)}(\Omega)} < 1 \ (\text{resp. } =1, >1) \ \text{if and only if } \varrho_{\Phi}(u) < 1 \ (\text{resp. } =1, >1),$
    \item[(iii)] $\min\big\{\|u\|^{p^-}_{L^{\Phi(\cdot, \cdot)}(\Omega)}, \|u\|^{q^++r^+} _{L^{\Phi(\cdot, \cdot)}(\Omega)}\big\}
    \leq \varrho_{\Phi}(u) 
    \leq \max\big\{\|u\|^{p^-}_{L^{\Phi(\cdot, \cdot)}(\Omega)}, \|u\|^{q^++r^+}_{L^{\Phi(\cdot, \cdot)}(\Omega)}\big\},$
    \item[(iv)] $\|u\|_{L^{\Phi(\cdot, \cdot)}(\Omega)} \to 0 \ \text{if and only if } \varrho_{{\Phi}}(u) \to 0,$
    \item[(v)] $\|u\|_{L^{\Phi(\cdot, \cdot)}(\Omega)} \to \infty \ \text{if and only if } \varrho_{{\Phi}}(u) \to \infty.$
\end{itemize}
\end{proposition}
\begin{proof}
    First, by Definition \ref{definition 2.10} and Lemma \ref{main lemma for increasing}, we know that $L^{\Phi(\cdot, \cdot)}(\Omega)$ is a separable, reflexive Banach space. \\
    For (i) and (ii), note that the function $\lambda \rightarrow \varrho_{{\Phi}}(\frac{u}{\lambda})$ with $\lambda\geq 0$ is continuous, convex and strictly increasing. This directly implies (i), and (i) with the strict increasing property yields (ii).\\
    Finally, (iii) follows from Proposition \ref{proposition 2.12} and Lemma \ref{main lemma for increasing}; and (iv) and (v) follow from (iii).
\end{proof}

\bigskip

Here are some basic embeddings of the space $L^{\Phi(\cdot, \cdot)}(\Omega)$, where $\Phi(x,t):=t^{p(x)}+a(x)t^{q(x)}(\log(e+t))^{r(x)}.$ In the following proposition, we write $L^{\zeta}(\Omega)=:L^{q(x)} \log L^{r(x)}(\Omega)$ for $\zeta(x,t) = t^{q(x)} \log^{r(x)}(e+t)$.

\begin{proposition}\label{proposition 2.19}
 Let  $\Omega \subset \mathbb R^n, n\geq 2,$ be a bounded domain and $ p, q\in C_+(\bar\Omega)$ with $ p(x)\leq q(x)$ for all $x\in\bar\Omega $ and $ \epsilon>0$, $r^+ \geq 0$ and $a \in L^1(\Omega)$; then it holds
\[
L^{q(\cdot) + \varepsilon r^+ }(\Omega) \hookrightarrow L^{q(\cdot)} \log L^{r(\cdot)}(\Omega)
\quad \text{and} \quad
L^{\Phi(\cdot, \cdot)}(\Omega) \hookrightarrow L^{p(\cdot)}(\Omega).
\]
If we further assume $a \in L^\infty(\Omega)$, then it also holds
\[
L^{q(\cdot)} \log L^{r(\cdot)}(\Omega) \hookrightarrow L^{{\Phi(\cdot, \cdot)}}(\Omega).
\]
\end{proposition}

\begin{proof}
We will prove all embeddings by applying Proposition \ref{proposition 2.13} to the corresponding
$\Phi$-functions.  

First, for any $K > 0$, by the inequality 
\[
\log^{r(x)}(e+t) \leq (C_\varepsilon + t^\varepsilon)^{r^+}\leq C_{r}(C_\varepsilon^{r^+} + t^{\varepsilon r^{+}})
\]
it holds that
\[
\left( \frac{t}{K} \right)^{q(x)} \log^{r(x)}\left(e + \frac{t}{K}\right) 
   \leq \frac{t^{q(x)+\varepsilon r^+} + 1}{K^{q(x)}} C_\varepsilon^{r^+} C_{r}
        + C_{r}\frac{t^{q(x)+\varepsilon r^{+}}}{K^{q(x)+\varepsilon r^{+}}}
\]
and if we choose $K \geq \max \{ (2C_{r})^{1/(q^- +\varepsilon r^+)}, (2C_rC_\varepsilon^{r^{+}} )^{1/q^-}, (C_\varepsilon^{r^{+}} C_{r} | \Omega |)^{1/q^-} \}$, it follows that
\[
\left( \frac{t}{K} \right)^{q(x)} \log^{r(x)}\left(e + \frac{t}{K}\right) 
   \leq \frac{C_\varepsilon^{r^{+}} C_{r} }{K^{q(x)}} 
        + t^{q(x)+\varepsilon r^+}, \, \, \text{with} \int_{\Omega} \frac{C_\varepsilon^{r^{+}} C_{r} }{K^{q(x)}} dx \leq 1.
\]
This concludes the proof of the first embedding. The condition for the second embedding is straightforward to verify. Finally, for any $K > 0$,
\[
\Phi\left(x,\frac{t}{K}\right) 
   \leq \frac{1}{K^{p(x)}} + \left( \frac{1}{K^{q(x)}} 
          + \frac{\|a\|_\infty}{K^{q(x)}} \right) t^{q(x)} \log^{r(x)}\!\left(e + \frac{t}{K}\right).
\]
If we choose $K \geq \max \{ 2^{1/p^-}, (2 \|a\|_\infty)^{1/q^-}, ( |\Omega |)^{1/p^-} \}$, it follows that
\[
\Phi\left(x,\frac{t}{K}\right) 
   \leq \frac{1}{K^{p(x)}} + t^{q(x)} \log^{r(x)}(e+t)
   \quad \text{with} \quad \int_\Omega \frac{1}{K^{p(x)}} \, dx \leq 1.
\]
This shows the proof of the third embedding.
\end{proof}

\bigskip
The proof of the next proposition is completely analogous to the proof of Proposition \ref{proposition 2.18} except that now we use Proposition \ref{proposition 2.15} and Proposition \ref{proposition 2.16}.

\begin{proposition}
 Let  $\Omega \subset \mathbb R^n, n\geq 2,$ be a domain. Assume that $ p, q\in C_+(\bar\Omega)$ with $ p(x)\leq q(x)$ for all $x\in\bar\Omega $, $r^+ \geq 0$ and $a \in L^1(\Omega)$. Then $W^{1,\Phi(\cdot,\cdot)} (\Omega)$ and $W_0^{1,\Phi(\cdot,\cdot)} (\Omega),$ where $\Phi(x,t)=t^{p(x)}+a(x)t^{q(x)}(\log(e+t))^{r(x)},$ are separable, reflexive Banach spaces and the following hold:
\begin{itemize}
    \item[(i)] $\|u\|_{W^{1,\Phi(\cdot,\cdot)} (\Omega)} = \lambda \ \text{if and only if } \varrho_{W^{1,\Phi(\cdot,\cdot)} (\Omega)}\!\left(\tfrac{u}{\lambda}\right) = 1 \ \text{for } u\neq 0, \lambda > 0,$
    \item[(ii)] $\|u\|_{W^{1,\Phi(\cdot,\cdot)} (\Omega)} < 1 \ (\text{resp. } =1, >1) \ \text{if and only if } \varrho_{W^{1,\Phi(\cdot,\cdot)} (\Omega)}(u) < 1 \ (\text{resp. } =1, >1),$
    \item[(iii)] $\min\big\{\|u\|^{p^-}_{W^{1,\Phi(\cdot,\cdot)} (\Omega)}, \|u\|^{q^++r^+}_{W^{1,\Phi(\cdot,\cdot)} (\Omega)}\big\} 
    \leq \varrho_{W^{1,\Phi(\cdot,\cdot)} (\Omega)}(u) 
    \leq \max\big\{\|u\|^{p^-}_{W^{1,\Phi(\cdot,\cdot)} (\Omega)}, \|u\|^{q^++r^+}_{W^{1,\Phi(\cdot,\cdot)} (\Omega)}\big\},$
    \item[(iv)] $\|u\|_{W^{1,\Phi(\cdot,\cdot)} (\Omega)} \to 0 \ \text{if and only if } \varrho_{W^{1,\Phi(\cdot,\cdot)} (\Omega)}(u)\to 0,$
    \item[(v)] $\|u\|_{W^{1,\Phi(\cdot,\cdot)} (\Omega)} \to \infty \ \text{if and only if } \varrho_{W^{1,\Phi(\cdot,\cdot)} (\Omega)}(u)\to \infty.$
\end{itemize}
\end{proposition}

We recall the following two results on Sobolev spaces of variable exponents from \cite{BGHW22}, which will be used to prove some elementary results on embedding.
\begin{proposition}\label{proposition 2.9}
Let $r \in C^{0,\frac{1}{|\log|}}(\overline{\Omega}) \cap C_{+}(\overline{\Omega})$ and let $s \in C(\overline{\Omega})$ be such that 
$1 \leq s(x) \leq r^{*}(x)$ for all $x \in \overline{\Omega},$ where $r^{*}(x)= \frac{nr(x)}{n-r(x)}$.
Then we have the continuous embedding
\[
W^{1,r(\cdot)}(\Omega) \hookrightarrow L^{s(\cdot)}(\Omega).
\]
If $r \in C_{+}(\overline{\Omega})$, $s \in C(\overline{\Omega})$ and
$1 \leq s(x) < r^{*}(x)$ for all $x \in \overline{\Omega},$
then the embedding is compact.
\end{proposition}

\begin{proposition}\label{proposition 2.10}
Suppose that $r \in C_{+}(\overline{\Omega}) \cap W^{1,\gamma}(\Omega)$ for some $\gamma > N$.  
Let $s \in C(\overline{\Omega})$ be such that 
$1 \leq s(x) \leq r_{*}(x)$ for all $x \in \overline{\Omega}$, where $r_{*}(x)=\frac{(n-1)r(x)}{n-r(x)}$. Then we have the continuous embedding
\[
W^{1,r(\cdot)}(\Omega) \hookrightarrow L^{s(\cdot)}(\partial \Omega).
\]
If $r \in C_{+}(\overline{\Omega})$, $s \in C(\overline{\Omega})$ and
$1 \leq s(x) < r_{*}(x)$ for all $x \in \overline{\Omega}$, then the embedding is compact.
\end{proposition}
These Sobolev spaces satisfy the following embeddings.
\begin{proposition}\label{elementary embedding}
Let  $\Omega \subset \mathbb R^n, n\geq 2,$ be a bounded domain with Lipschitz boundary $\partial\Omega.$ Take $\Phi(x,t)=t^{p(x)}+a(x)t^{q(x)}(\log(e+t))^{r(x)},$ where $ p, q\in C_+(\bar\Omega)$ with $ p(x)\leq q(x)$ for all $x\in\bar\Omega $, $r^+ \geq 0$ and $a \in L^1(\Omega)$. Then the following hold:
\begin{itemize}
    \item[(i)] ${W^{1,\Phi(\cdot,\cdot)} (\Omega)} \hookrightarrow W^{1,p(\cdot)}(\Omega)$ and ${W_0^{1,\Phi(\cdot,\cdot)}}(\Omega) \hookrightarrow W^{1,p(\cdot)}_0(\Omega)$ are continuous;
    \item[(ii)] If $p \in C_{+}(\overline{\Omega}) \cap C^{0,\frac{1}{|\log t|}}(\overline{\Omega})$, then ${W^{1,\Phi(\cdot,\cdot)} (\Omega)} \hookrightarrow L^{p^*(\cdot)}(\Omega)$ and ${W_0^{1,\Phi(\cdot,\cdot)} }(\Omega) \hookrightarrow L^{p^*(\cdot)}(\Omega)$ are continuous;
    \item[(iii)] $W^{1,\Phi(\cdot,\cdot)} (\Omega) \hookrightarrow L^{r(\cdot)}(\Omega)$ and $W_0^{1,\Phi(\cdot,\cdot)} (\Omega) \hookrightarrow L^{r(\cdot)}(\Omega)$ are compact for $r \in C(\overline{\Omega})$ with $1 \le r(x) < p(x)$ for all $x \in \overline{\Omega}$;
    \item[(iv)] if $p \in C_{+}(\overline{\Omega}) \cap W^{1,\gamma}(\Omega)$ for some $\gamma> N$, then $W^{1,\Phi(\cdot,\cdot)} (\Omega) \hookrightarrow L^{p^*(\cdot)}(\Omega)$ and $W_0^{1,\Phi(\cdot,\cdot)} (\partial\Omega) \hookrightarrow L^{p^*(\cdot)}(\partial\Omega)$ are continuous;
    \item[(v)] $W^{1,\Phi(\cdot,\cdot)} (\Omega) \hookrightarrow L^{r(\cdot)}(\Omega)$ and $W_0^{1,\Phi(\cdot,\cdot)} (\Omega) \hookrightarrow L^{r(\cdot)}(\partial\Omega)$ are compact for $r \in C(\overline{\Omega})$ with $1 \le r(x) < p^*(x)$ for all $x \in \overline{\Omega}$.
\end{itemize}
\end{proposition}
\begin{proof}
    The proof of (i) follows directly from Proposition \ref{proposition 2.19}. The proofs of (ii)-(v) follow from (i) and the usual Sobolev embeddings of $W^{1,p(\cdot)}(\Omega)$ and $W_0^{1,p(\cdot)}(\Omega)$ in Propositions \ref{proposition 2.9} and \ref{proposition 2.10}.
\end{proof}

The following lemma is a consequence of the log-H\"older continuity condition.
\begin{lemma}{\cite[Lemma 4.1.6]{DHHR11}}\label{logholder continuous}
     Suppose that $p:\mathbb{R}^n\rightarrow \mathbb{R}$ satisfies the log-H\"older continuity condition $(p1)$. Let $B \subset  \mathbb{R}^n$ be any ball of radius $r$ and $x, y\in B$ such that $p(x)\geq p(y)$. Then we have $r^{p(y)-p(x)} \leq c.$
\end{lemma}

    Sufficient conditions for smooth functions to be dense in a Musielak-Orlicz-Sobolev space are given in the following result from Theorem 6.4.7 of the book by Harjulehto-H\"ast\"o \cite{HH19}.
    \begin{theorem} \label{density theorem}
      Let $\Phi:\Omega\times[0,\infty)\rightarrow[0,\infty]$  be a generalized weak $\Phi$-function that satisfies $(A0), (A1), (A2)$ and $(aDec).$ Then $C^{\infty}(\Omega)\cap W^{1,\Phi(\cdot,\cdot)}(\Omega) $ is dense in $W^{1,\Phi(\cdot,\cdot)}(\Omega).$
    \end{theorem}
   \noindent In the next section, we will show that our function $\Phi(\cdot, \cdot)$ satisfies $(A0), (A1)$ and $(A2),$ and hence the denseness of smooth functions in $W^{1,\Phi(\cdot,\cdot)}(\Omega)$ for this function will follow from Lemma \ref{main lemma for increasing} and the above theorem.

    \section{Sufficient Part}
   In this section, we will show conditions $(A0), (A1)$ and $ (A2)$ for the function $\Phi(x,t)=t^{p(x)}+a(x)t^{q(x)}(\log(e+t))^{r(x)}$ with different assumptions on the exponents $p, q$ and $r.$
   \begin{lemma}\label{lemma for (A0)}
       Let $\Omega \subset \mathbb R^n, n\geq 2,$ be a domain. Assume that $ p, q, r\in C(\bar\Omega)$ with $1\leq p(x)\leq q(x)$ for all $x\in\bar\Omega $, $r^+ \geq 0,$ and $a \in L^1(\Omega)$. Then $\Phi(x,t)$ satisfies $(A0)$.
   \end{lemma}
   \begin{proof}
       By Lemma \ref{equivalent lemma}(i), it is equivalent to checking the condition $(A0)'$. 
       If we take $\beta^{-1} = 2(1+||a||_{\infty})\log^{r^+}(e+1/2),$ then
    \begin{equation*}
        \Phi(x, \beta) \leq \frac{1}{2}+ \frac{1}{2}\frac{||a||_{\infty}}{1+||a||_{\infty}}\frac{\log^{r^+}(e+1/2)}{\log^{r^+}(e+1/2)} \leq 1
       \end{equation*} 
and 
\begin{equation*}
      \Phi(x, \beta^{-1})\geq 2 +2a(x)\log^{r^+}(e+2) \geq 1,
    \end{equation*} 
    which shows that $\Phi$ satisfies the condition $(A0).$
   \end{proof}
    
    \begin{lemma}\label{lemma for (A1)}
      Let $\Omega \subset \mathbb R^n, n\geq 2,$ be a bounded domain. Suppose that $p, q:\bar\Omega\rightarrow [1, \infty)$ and  $a:\bar\Omega\rightarrow [0, \infty)$ are H\"older continuous functions with $1\leq p(x)\leq q(x)$ for all $x\in \bar \Omega$ and 
      \begin{equation*}
       \big(\frac{q}{p} \big)^{+}  < 1+ \frac{\gamma}{n},
      \end{equation*}
          
     \noindent where $\gamma$ is the  H\"older exponent of $a.$ If the function $r:\bar\Omega\rightarrow [0, \infty)$ is log-log-H\"older continuous, then $\Phi(x,t)$ satisfies $(A1)$. 
    \end{lemma}
    \begin{proof}
         To show that $\Phi$ satisfies $(A1)$, by Lemma \ref{equivalent lemma} (ii), it is equivalent to checking $(A1)'.$ Let $B \subset \mathbb R^n$ be any ball with $|B|\leq 1.$ 
         From the condition $\Phi(y,t)\in [1,\frac{1}{|B|}]$, we have that
        if $t\leq 1,$  
        $$1\leq \Phi(y,t)\leq [1+||a||_{\infty}\log^{r^+}(e+1)]t^{p(y)},$$
        and if $t\geq 1$ then the inequality $1 \leq [1+||a||_{\infty}\log^{r^+}(e+1)]t^{p(y)}$ holds trivially. Hence, in any case, we finally have
        \begin{equation}\label{equation (3.1)}
            \frac{1}{[1+||a||_{\infty}\log^{r^+}(e+1)]^{\frac{1}{p(y)}}}\leq t\leq \frac{1}{|B|^{\frac{1}{p(y)}}}.
        \end{equation}
        Claim 1: There is a positive constant $M$ that depends only on $n, p, q, r$ and $a$ such that 
        \begin{equation*}
            t^{p(x)}\leq Mt^{p(y)} \, \,  \text{and} \, \, t^{q(x)}\leq Mt^{q(y)}
        \end{equation*}
        hold for all $x,y\in B\cap\Omega,$ $t\geq 0$ with $\Phi(y,t)\in [1, 1/|B|]$, and any ball $B \subset \mathbb R^n$ with $|B|\leq 1.$\\
        
        Proof of Claim 1: We will prove only the first inequality, as the other one will follow in a similar fashion. The cases when $t\leq 1,$ $p(x)\geq p(y)$ or $t\geq 1,$ $p(x)\leq p(y)$ are trivial as one can simply take $M=1.$ Now for the case when $t\leq 1,$ $p(x)\leq p(y),$ we get from \eqref{equation (3.1)} that
        \begin{equation*}
           t^{p(x)} = t^{p(x)-p(y)}t^{p(y)}\leq ([1+||a||_{\infty}\log^{r^+}(e+1)]^{\frac{1}{p(y)}})^{p(y)-p(x)} t^{p(y)} \leq [1+||a||_{\infty}\log^{r^+}(e+1)]t^{p(y)}, 
        \end{equation*}
        and hence one can take $M=[1+||a||_{\infty}\log^{r^+}(e+1)] $. Finally, we consider the case $t\geq 1,$ $p(x)\geq p(y).$ 
        Suppose that $B$ is any ball of radius $R$ with $|B|=\omega(n)R^n\leq 1.$ Since $p$ is H\"older continuous, there exist constants $0<\alpha\leq 1$ and $c_p>0$ such that for any $x,y\in B$ we have $|p(x)-p(y)|\leq c_p 2^{\alpha} R^{\alpha}.$ 
        Since $|B|\leq 1,$ we can write $|B|^{-1/p(y)}\leq |B|^{-1/p^-}$ and hence using \eqref{equation (3.1)} we get 
        \begin{equation*}
            t^{p(x)} = t^{p(x)-p(y)}t^{p(y)}\leq (\omega(n)R^n)^{\frac{-c_p 2^{\alpha} R^{\alpha}}{p^-}} t^{p(y)} = (\omega(n)^{\frac{-c_p 2^{\alpha}}{p^-}})^{R^{\alpha}} (R^{R^{\alpha}})^{\frac{-c_p 2^{\alpha}n}{p^-}} t^{p(y)}.
        \end{equation*}
        Now consider the function 
        $h(R) := (\omega(n)^{\frac{-c_p 2^{\alpha}}{p^-}})^{R^{\alpha}} (R^{R^{\alpha}})^{\frac{-c_p 2^{\alpha}N }{p^-}}$ if $R \neq 0$ and $h(0):= 1$. This function is strictly positive and continuous in the interval $[0, \omega(n)^{-1/n}].$ Thus, it attains its maximum at some point $R_0$ in that interval, and we can take $M= h(R_0).$ This ends the proof of claim 1.\\
        
        Claim 2: There is a positive constant $N$ that depends only on $n,$ $p$ and $r$ such that 
        \begin{equation*}
            (\log(e+t))^{r(x)}\leq N(\log(e+t))^{r(y)}
        \end{equation*}
         holds for all $x,y\in B\cap\Omega,$ $t\geq 0$ with $\Phi(y,t)\in [1, 1/|B|]$, and any ball $B \subset \mathbb R^n$ with $|B|\leq 1.$\\
        
        Proof of Claim 2: If $r(x)\leq r(y)$ then one can simply take $N=1$. Consider the other case $r(x)\geq r(y)$. Since $r(\cdot)$ is log-log-H\"older continuous, we get, using \eqref{equation (3.1)},
        \begin{eqnarray*}
            \log^{r(x)}(e+t)&=& \log^{r(x)-r(y)}(e+t)\log^{r(y)}(e+t)\\
            &\leq& (\log(e+\frac{1}{|B|^{\frac{1}{p^-}}}))^{\frac{C}{ \log(e+\log(e+\frac{1}{|x-y|}))}}\log^{r(y)}(e+t)\\
            &\leq& e^{\frac{c_0\log(\log(e+\frac{1}{|x-y|^n}))}{\log(e+\log(e+\frac{1}{|x-y|}))}}\log^{r(y)}(e+t),
        \end{eqnarray*}
  where $c_0$ is a constant depending on $n$ and $p.$\\
  
   Now, if $|x-y|\leq 1$ then we have $\log(\log(e+\frac{1}{|x-y|}))> \log(\log(e+1))$ and hence
   \begin{center}
   $e^{\frac{c_0\log(\log(e+\frac{1}{|x-y|^n}))}{\log(e+\log(e+\frac{1}{|x-y|}))}} \leq e^{\frac{c_0(\log n +\log(\log(e+\frac{1}{|x-y|})))}{\log(e+\log(e+\frac{1}{|x-y|}))}}\leq e^{ c_0+\frac{c_0\log n}{\log(\log(e+1))}}.$
   \end{center}
   On the other hand, if $|x-y|>1,$ then $\log(\log(e+\frac{1}{|x-y|^n})\leq \log(e+\log(e+\frac{1}{|x-y|}))$ and hence 
   \begin{equation*}
   e^{\frac{c_0\log(\log(e+\frac{1}{|x-y|^n}))}{\log(e+\log(e+\frac{1}{|x-y|}))}} \leq e^{c_0}.
   \end{equation*}
   
   Therefore, setting $N = e^{ c_0+\frac{c_0\log n}{\log(\log(e+1))}}$ we complete the proof of claim 2.\\
   

   Now, let us come back to the proof of $(A1)'.$ Take $0<\beta<(MN)^{-1/p^-}<1.$ Since $\gamma$ is the H\"older exponent of $a,$ there is a constant $c_a>0$ such that for any $x, y \in B$ we have $|a(x)-a(y)|\leq c_a 2^{\gamma}R^{\gamma}.$ Then Claim 1 and Claim 2 yield 
   \begin{eqnarray*}
       \Phi(x,\beta t)&\leq& \beta^{p^-}\big(t^{p(x)}+a(x)t^{q(x)}(\log(e+t))^{r(x)}\big)\\
        &\leq& \beta^{p^-}M\big(t^{p(y)}+Na(x)t^{q(y)}(\log(e+t))^{r(y)}\big)\\
        &\leq& \beta^{p^-}M\big(t^{p(y)}+Na(y)t^{q(y)}(\log(e+t))^{r(y)} +c_a2^{\gamma}R^{\gamma}Nt^{q(y)}(\log(e+t))^{r(y)} \big)\\
        &\leq&  \beta^{p^-}M\big(t^{p(y)} +c_a2^{\gamma}R^{\gamma}Nt^{q(y)}(\log(e+t))^{r^+} \big)+a(y)t^{q(y)}(\log(e+t))^{r(y)}.
   \end{eqnarray*}
   Using \eqref{equation (3.1)} and the fact that $|B|= \omega(n)R^n \leq 1$, we obtain
   \begin{eqnarray*}
       \Phi(x,\beta t)&\leq & \beta^{p^-}Mt^{p(y)}\big(1 +c_a2^{\gamma}R^{\gamma}Nt^{q(y)-p(y)}(\log(e+(\omega(n)R^n)^{\frac{-1}{p(y)}}))^{r^+} \big)\\
        & &+a(y)t^{p(y)}(\log(e+t))^{r(y)} \\
       &\leq & \beta^{p^-}Mt^{p(y)}[\big(1 +c_a2^{\gamma}R^{\gamma}N(\omega(n)R^n)^{\frac{-(q(y)-p(y))}{p(y)}}(\log(e+(\omega(n)R^n)^{\frac{-1}{p^-}}))^{r^+} \big)]\\
        & &+a(y)t^{p(y)}(\log(e+t))^{r(y)} \\
       &\leq & \beta^{p^-}MNt^{p(y)}[\big(1 +c_a2^{\gamma}R^{\gamma}(\omega(n)R^n)^{\frac{-(q(y)-p(y))}{p(y)}}(\log(e+(\omega(n)R^n)^{\frac{-1}{p^-}}))^{r^+} \big)]\\
       & &+a(y)t^{p(y)}(\log(e+t))^{r(y)}.
   \end{eqnarray*}
  Let us concentrate on the first part of the above sum in the right hand side. Consider $\tau_{p,q,n}= \frac{q_{-}}{p_{+}}$ if $\omega(n)>1$ and $\tau_{p,q,n}= \frac{q_{+}}{p_{-}}$ if $\omega(n)\leq 1.$ Once again, as $\omega(n)R^n\leq 1,$ 
   \begin{equation*}
       R^{\gamma}(\omega(n)R^n)^{\frac{-(q(y)-p(y))}{p(y)}}(\log(e+(\omega(n)R^n)^{\frac{-1}{p^-}}))^{r^+}\leq \omega(n)^{1-\tau_{p,q,n}}R^{\gamma+n-\frac{nq(y)}{p(y)}}(\log(e+(\omega(n)R^n)^{\frac{-1}{p^-}}))^{r^+}.
   \end{equation*}
If $R\leq 1,$ then
   \begin{equation*}
       R^{\gamma+n-\frac{nq(y)}{p(y)}}(\log(e+(\omega(n)R^n)^{\frac{-1}{p^-}}))^{r^+}\leq R^{\gamma+n-n\big(\frac{q}{p} \big)_{+}}(\log(e+(\omega(n)R^n)^{\frac{-1}{p^-}}))^{r^+}=: g(R).
   \end{equation*}

 This function $g$ with $g(0)=0$ satisfies \begin{equation*}
       \lim_{R\rightarrow 0^+} g(R)=0,
   \end{equation*}
   as $\gamma+n-n\big(\frac{q}{p} \big)_{+}>0,$ which means that the function $g$ is positive and continuous in the interval $[0, \omega(n)^{-1/n}]$. Hence, the function g reaches its maximum at some point $R_0$ in that interval, and we can use $g(R_0)$ as the upper estimate. In the other case, if $R\geq 1,$
   \begin{eqnarray*}
       R^{\gamma+n-\frac{nq(y)}{p(y)}}(\log(e+(\omega(n)R^n)^{\frac{-1}{p^-}}))^{r^+}&\leq& R^{\gamma+n-n\big(\frac{q}{p} \big)_{-}}(\log(e+(\omega(n))^{\frac{-1}{p^-}}))^{r^+}\\ 
       &\leq & \omega(n)^{-\frac{\gamma}{n}-1+\big(\frac{q}{p} \big)_{-}}(\log(e+(\omega(n))^{\frac{-1}{p^-}}))^{r^+}\\
       &=:&\tilde{S}_{p,q,n,r},
   \end{eqnarray*}
   which follows from $\gamma+n-n\big(\frac{q}{p} \big)_{-}>0$ and $R\leq \omega(n)^{\frac{-1}{n}}.$\\
   Denote ${S}_{p,q,n,r} =: \max \{ \omega(n)^{1-\tau_{p,q,n}} g(R_0), \, 
      \omega(n)^{1-\tau_{p,q,n}}\tilde{S}_{p,q,n,r}\}$. Together, we have proved that 
   \begin{equation*}
        \Phi(x,\beta t)\leq \beta^{p^-}MNt^{p(y)}[\big(1 +c_a2^{\gamma}{S}_{p,q,n,r} \big)]+a(y)t^{q(y)}(\log(e+t))^{r(y)}.
   \end{equation*}
   If we take 
   \begin{equation*}
       \beta < (MN)^{\frac{1}{p^-}} [1 +c_a2^{\gamma}{S}_{p,q,n,r}]^{\frac{-1}{p^-}},
   \end{equation*}
   we obtain $\Phi(x,\beta t)\leq \Phi(y, t)$ and the proof is complete.   
    \end{proof}
   
    Now we will show that a weaker condition, namely the log-H\"older continuity of $p,q$ is enough to show that $\Phi(x,t):= t^{p(x)}+a(x)t^{q(x)}{(\log(e+t))}^{r(x)}$ satisfies $(A1)$ if $r(x)\geq 0$ and we also do not need log-log-H\"older continuity assumption on $r(x).$ \\
    \begin{proposition}\label{phi satisfies all the condition}
Let $\Omega \subset \mathbb R^n, n\geq 2,$ be a domain. Assume that $0\leq a(\cdot)\in L^{\infty}(\Omega)$ and $p:\bar\Omega\rightarrow [1, \infty),$ $q:\bar\Omega\rightarrow [1, \infty)$ are log-H\"older continuous with $p(x)\leq q(x)$ for all $x\in \bar\Omega.$ Then the function $\Phi(x,t):= t^{p(x)}+a(x)t^{q(x)}{(\log(e+t))}^{r(x)}$ satisfies $ (A1)$ for \,    $r(x)\geq 0$ if and only if there exists a constant $\beta >0$ such that 
\begin{center}
         $\beta a(y)^{\frac{1}{q(y)}} \leq  a(x)^{\frac{1}{q(x)}} + \frac{1}{\log(e+|x-y|^{-1})}$.
     \end{center}
\end{proposition}
\begin{proof}
     Let us first note that $\Phi(x,t) \approx \max \{t^{p(x)}, a(x)t^{q(x)}{(\log(e+t))}^{r(x)}\}$ and hence 
     \begin{center}
         $\Phi^{-1}(x,t) \approx \min \{ t^{\frac{1}{p(x)}}, \frac{t^{\frac{1}{q(x)}}}{a(x)^{\frac{1}{q(x)}}(\log(e+\frac{t}{a(x)}))^{\frac{r(x)}{q(x)}}}\}$.
     \end{center}
Hence for the function $\Phi,$ condition $(A1)$ becomes, after dividing by $t^{\frac{1}{q(y)}},$  
     \begin{center}
         $\beta \min \{ t^{\frac{1}{p(x)}-\frac{1}{q(y)}}, \frac{t^{\frac{1}{q(x)}-\frac{1}{q(y)}}}{a(x)^{\frac{1}{q(x)}}(\log(e+\frac{t}{a(x)}))^{\frac{r(x)}{q(x)}}}\}\leq \min \{ t^{\frac{1}{p(y)}-\frac{1}{q(y)}}, \frac{1}{a(y)^{\frac{1}{q(y)}}(\log(e+\frac{t}{a(y)}))^{\frac{r(y)}{q(y)}}}\}$
     \end{center}
     which is the same as
      \begin{center}
         $\beta \min \{ t^{\frac{1}{p(y)}-\frac{1}{q(y)}}, \frac{t^{\frac{1}{q(x)}-\frac{1}{q(y)}}}{a(x)^{\frac{1}{q(x)}}(\log(e+\frac{t}{a(x)}))^{\frac{r(x)}{q(x)}}}\}\leq \min \{ t^{\frac{1}{p(y)}-\frac{1}{q(y)}}, \frac{1}{a(y)^{\frac{1}{q(y)}}(\log(e+\frac{t}{a(y)}))^{\frac{r(y)}{q(y)}}}\}$
     \end{center}
     for  $x,y \in B\cap\Omega, |B|\leq 1$ and $t\in [1, \frac{1}{|B|}]$.\\
     If $p(x)\leq p(y)$, then 
     \begin{center}
        $t^{\frac{1}{p(x)}} \leq t^{\frac{1}{p(x)}-\frac{1}{p(y)}} t^{\frac{1}{p(y)}} \leq |B|^{\frac{1}{p(y)}-\frac{1}{p(x)}}t^{\frac{1}{p(y)}} \leq Ct^{\frac{1}{p(y)}} $,
     \end{center}
     where we have used the $\log$-H\"older continuity of $\frac{1}{p}$ and $t\in [1, \frac{1}{|B|}],$ see Lemma \ref{logholder continuous}. If $p(x)\geq  p(y)$, then $t^{\frac{1}{p(x)}}\leq t^{\frac{1}{p(y)}}$ since $t\geq 1$. So $t^{\frac{1}{p(x)}} \leq C_1 t^{\frac{1}{p(y)}}$ where $C_1= \min(1, C)$. Similarly, we can also show that $t^{\frac{1}{q(x)}-\frac{1}{q(y)}} \leq C_2 $.\\
     Thus, for condition $(A1)$ we need to show that 
     
     \begin{center}
         $\beta \min \{ t^{\frac{1}{p(y)}-\frac{1}{q(y)}}, \frac{1}{a(x)^{\frac{1}{q(x)}}(\log(e+\frac{t}{a(x)}))^{\frac{r(x)}{q(x)}}}\}\leq \min \{ t^{\frac{1}{p(y)}-\frac{1}{q(y)}}, \frac{1}{a(y)^{\frac{1}{q(y)}}(\log(e+\frac{t}{a(y)}))^{\frac{r(y)}{q(y)}}}\}$.
     \end{center}
     Aside from the trivial case $t^{\frac{1}{p(y)}-\frac{1}{q(y)}} \leq  \frac{1}{a(y)^{\frac{1}{q(y)}}(\log(e+\frac{t}{a(y)}))^{\frac{r(y)}{q(y)}}}$, we get the equivalent condition 
     \begin{equation}\label{equation 1}
         \beta' a(y)^{\frac{1}{q(y)}}(\log(e+\frac{t}{a(y)}))^{\frac{r(y)}{q(y)}} \leq \max \{ t^{\frac{1}{q(y)}-\frac{1}{p(y)}}, a(x)^{\frac{1}{q(x)}}(\log(e+\frac{t}{a(x)}))^{\frac{r(x)}{q(x)}}\}.
     \end{equation}
     Denote $f(u)= u^{\frac{1}{q(y)}}(\log(e+\frac{1}{u}))^{\frac{r(y)}{q(y)}}$, $g(u) =\frac{u^{q(y)}}{(\log(e+\frac{1}{u}))^{r(y)}} $, $h(u)=u^{\frac{1}{q(x)}}(\log(e+\frac{1}{u}))^{\frac{r(x)}{q(x)}} $ and $I(u) =\frac{u^{q(x)}}{(\log(e+\frac{1}{u}))^{r(x)}}$. Then 
     \begin{center}
         $g(bf(u)) \approx u$
     \end{center}
     where the implicit constant depends on b. We can write \eqref{equation 1} as 
     \begin{equation}
         \beta' t^{\frac{1}{q(y)}} f(\frac{a(y)}{t})\leq \max \{t^{\frac{1}{q(y)}-\frac{1}{p(y)}}, t^{\frac{1}{q(x)}} h(\frac{a(x)}{t})\}.
     \end{equation}
     Now, dividing both sides by $ t^{\frac{1}{q(y)}}$ we get 
     \begin{equation}\label{equation 2}
         \beta'  f(\frac{a(y)}{t})\leq \max \{t^{\frac{-1}{p(y)}}, t^{\frac{1}{q(x)}-{\frac{1}{q(y)}}} h(\frac{a(x)}{t})\}. 
     \end{equation}
     Now, applying the increasing function $g$ to both sides of (\ref{equation 2}), we get the following:
     \begin{center}
         $\frac{a(y)}{t}\approx g(\beta'  f(\frac{a(y)}{t}))\leq  \max \{g(t^{\frac{-1}{p(y)}}), g(t^{\frac{1}{q(x)}-{\frac{1}{q(y)}}}h(\frac{a(x)}{t}))\}$
     \end{center}
     or, equivalently,
     \begin{eqnarray*}
         & &
         \beta a(y) \leq t g(t^{\frac{-1}{p(y)}})+ tg(t^{\frac{1}{q(x)}-{\frac{1}{q(y)}}}h(\frac{a(x)}{t}))\\ &\lesssim& \frac{1}{(\log(e+t^{\frac{1}{p(y)}}))^{r(y)}}+ q(x)^{r(y)}a(x)^{\frac{q(y)}{q(x)}}(\log(e+\frac{t}{a(x)}))^{\frac{r(x)q(y)}{q(x)}}(\log(e+\frac{t^{\frac{q(x)}{q(y)}}}{a(x)}))^{-r(y)}\\
         &\approx& \frac{1}{(\log(e+t^{\frac{1}{p(y)}}))^{r(y)}}+ q(x)^{r(y)}a(x)^{\frac{q(y)}{q(x)}}(\log(e+\frac{t}{a(x)}))^{\frac{r(x)q(y)-r(y)q(x)}{q(x)}},
     \end{eqnarray*}
     where in the final inequality, we have used $t\approx  t^{\frac{q(x)}{q(y)}}$ which follows from log-H\"older continuity of $q(x)$.     
     Therefore,
     \begin{center}
      $\beta a(y)^{\frac{1}{q(y)}}\leq \frac{1}{(\log(e+t^{\frac{1}{p(y)}}))^{\frac{r(y)}{q(y)}}}+ q(x)^{\frac{r(y)}{q(y)}}a(x)^{\frac{1}{q(x)}}(\log(e+\frac{t}{a(x)}))^{\frac{r(x)}{q(x)}-\frac{r(y)}{q(y)}}. $
      \end{center}
     By symmetry, we may assume that $\frac{r(x)}{q(x)}>\frac{r(y)}{q(y)}$. Then we get the equivalent condition
      \begin{eqnarray*}
         \beta a(y)^{\frac{1}{q(y)}} &\leq& \frac{1}{(\log(e+t^{\frac{1}{p(y)}}))^{\frac{1}{q(y)}}}+ q(x)^{\frac{r(y)}{q(y)}}a(x)^{\frac{1}{q(x)}}.
     \end{eqnarray*}
     Since we have $t\leq \frac{1}{|B|}\leq \frac{c}{|x-y|^n}$, we get the conditions as follows:
     \begin{center}
         $\beta a(y)^{\frac{1}{q(y)}}   \lesssim {(\log(e+|x-y|^{\frac{-1}{p(y)}}))^{-\frac{1}{q(y)}}}+ a(x)^{\frac{1}{q(x)}}, $
     \end{center}
    and hence finally, we get the equivalent condition given by:
    \begin{equation*}
    \beta a(y)^{\frac{1}{q(y)}} \leq  a(x)^{\frac{1}{q(x)}} + \frac{1}{\log(e+|x-y|^{-1})}. \qedhere
    \end{equation*} 
     \end{proof}
 \noindent As a special case, we get the following lemma.    
\begin{lemma}\label{weaker assumption on p}
    Let $\Omega \subset \mathbb R^n, n\geq 2,$ be a domain. Assume that $0\leq a(\cdot)\in L^{\infty}(\Omega)$ and $p:\bar\Omega\rightarrow [1, \infty),$ $q:\bar\Omega\rightarrow [1, \infty)$ are log-H\"older continuous with $p(x)\leq q(x)$ for all $x\in \bar\Omega$ and also, $a:\bar\Omega \rightarrow [0, \infty)$ is $q^+$-H\"older continuous. Then $\Phi(x, t):= t^{p(x)}+a(x)t^{q(x)}{\log^{r(x)}(e+t)}$, where $r(x)\geq 0$, satisfies $(A1)$.
\end{lemma}
\begin{proof}
    We will show that the inequality
    \begin{equation*}
        |a(x)^{\frac{1}{q(x)}} - a(y)^{\frac{1}{q(y)}}| \leq \frac{C}{\log(e+|x-y|^{-1})}
    \end{equation*}
    holds for all $x$, $y$ with $|x-y|\leq 1,$ and then the condition $(A1)$ will follow by the Lemma \ref{phi satisfies all the condition}. For this purpose, we first use the triangle inequality to obtain the following
    \begin{eqnarray}
        |a(x)^{\frac{1}{q(x)}} - a(y)^{\frac{1}{q(y)}}| &=& |a(x)^{\frac{1}{q(x)}} - a(x)^{\frac{1}{q(y)}} + a(x)^{\frac{1}{q(y)}} - a(y)^{\frac{1}{q(y)}}| \nonumber\\
       & \leq & |a(x)^{\frac{1}{q(x)}} - a(x)^{\frac{1}{q(y)}}| + |a(x)^{\frac{1}{q(y)}} - a(y)^{\frac{1}{q(y)}}|. \label{(A1) equation}
    \end{eqnarray}
    We estimate the first term on the right-hand side of \eqref{(A1) equation}. For the function $f(t)= b^t,$ we apply the mean value theorem to obtain $f(v)-f(u)= f'(\eta) (v-u)$ for some $\eta$ between u and v. We choose $b=a(x), u= \frac{1}{q(x)}$ and $v=\frac{1}{q(y)} $. Then $b\in [0, ||a||_{\infty}]$ and $u, v\in [\frac{1}{q^+}, 1].$\\
    Next, we will show that $|f'(\eta)|$ is bounded. If $b\geq 1$, then 
    \begin{equation*}
        |f'(\eta)|= b^{\eta}\log(b)\leq ||a||_{\infty} \log(||a||_{\infty}). 
    \end{equation*}
    For $a\in [0, 1)$, we obtain the following
    \begin{equation*}
        |f'(\eta)|= -b^{\eta}\log(b)\leq -b^{\frac{1}{q^+}} \log(b).
    \end{equation*}
    An easy calculation shows that the function $b\rightarrow -b^{\frac{1}{q^+}} \log(b)$ obtained its highest value in $[0, 1]$ at $e^{-q^+}.$ Hence $|f'(\eta)|\leq \frac{q^+}{e}$. Thus we have 
    \begin{equation}\label{(A1) 1equation}
        |a(x)^{\frac{1}{q(x)}} - a(x)^{\frac{1}{q(y)}}|= |f(v)-f(u)|\leq c |\frac{1}{q(x)} - \frac{1}{q(x)}|\leq c|q(y)-q(x)|\leq \frac{c}{\log(e+|x-y|^{-1})},
    \end{equation}
    where c is a constant depending on $b$ and $q$.\\
    Now we will estimate the second term of the right-hand side of \eqref{(A1) equation}. We use the inequality $|x^k-y^k|\leq |x-y|^k,$ where $x, y \geq 0$ and $k\in (0, 1]$, to obtain
    \begin{equation}\label{(A1) 2equation}
        |a(x)^{\frac{1}{q(y)}} - a(y)^{\frac{1}{q(y)}}|\leq |a(x)-a(y)|^{\frac{1}{q(y)}}\leq c_{a}^{\frac{1}{q^-}} |x-y|^{\frac{q^+}{q(y)}}\leq c_{a}^{\frac{1}{q^-}} |x-y| \leq \frac{c_{a}^{\frac{1}{q^-}}}{\log(e+|x-y|^{-1})},
    \end{equation}
    where the constant $c_a$ is from the $q^+$-H\"older continuity of $a(x)$.
    Therefore, using \eqref{(A1) 1equation} and \eqref{(A1) 2equation} in \eqref{(A1) equation} we get
    \begin{equation*}
        |a(x)^{\frac{1}{q(x)}} - a(y)^{\frac{1}{q(y)}}|\leq \frac{c_{a}^{\frac{1}{q^-}}+c}{\log(e+|x-y|^{-1})} =\frac{C}{\log(e+|x-y|^{-1})}.
    \end{equation*}
   Now, if $C\leq 1$, then $\beta =1$ works for both cases. If $C>1$, then there will be two cases.\\

Case I: If $a(x)^{\frac{1}{q(x)}} \leq  a(y)^{\frac{1}{q(y)}}$ then $\frac{a(y)^{\frac{1}{q(y)}}}{C}\leq \frac{a(x)^{\frac{1}{q(x)}}}{C} +\frac{1}{\log(e+|x-y|^{-1})}\leq a(x)^{\frac{1}{q(x)}}+\frac{1}{\log(e+|x-y|^{-1})}.$ If we take $\beta= 1/C$ then we are done.\\

Case II: If $a(x)^{\frac{1}{q(x)}} > a(y)^{\frac{1}{q(y)}}$ then $ a(y)^{\frac{1}{q(y)}}< a(x)^{\frac{1}{q(x)}} +\frac{1}{\log(e+|x-y|^{-1})}.$ Taking $\beta= 1,$ we complete the proof.
 \end{proof}
 \noindent  Now we will show that $\Phi$ also satisfies condition $(A2)$ for both bounded and unbounded domains.  
   
\begin{lemma}\label{lemma for (A2)}
   Let $\Omega \subset \mathbb R^n, n\geq 2,$ be a bounded domain and $ p, q, r\in C{(\bar\Omega)}$ with $1\leq p(x)\leq q(x)$ for all $x\in\bar\Omega $, $r^+ \geq 0$ and $a \in L^1(\Omega)$. Then $\Phi(x,t):= t^{p(x)}+a(x)t^{q(x)}{\log^{r(x)}(e+t)}$ satisfies $(A2)$.
\end{lemma}
\begin{proof}
    By Lemma \ref{equivalent lemma} (iii), it is equivalent to checking condition $(A2)'$.
    For $(A2)',$ take $s=1, \phi_{\infty}(t)= t^{p^+ +1}$ and $\beta=1.$ First note that $ \phi_{\infty}(t)\leq 1$ implies $t\leq 1.$ Thus, by Young's inequality 
    \begin{eqnarray*}
        \Phi(x, \beta t)&\leq& [1+||a||_{\infty}\log^{r^+}(e+1)]t^{p(x)} \\
        &\leq& \frac{p(x)}{p^+ +1}t^{p^+ +1} + \frac{p^+ -p(x)+1}{p^+ +1}[1+||a||_{\infty}\log^{r^+}(e+1)]^{\frac{p^+ +1}{p^+ -p(x)+1}}\\
        &\leq& \phi_{\infty}(t)+ [1+||a||_{\infty}\log^{r^+}(e+1)]^{p^+ +1}.
    \end{eqnarray*}
    Taking $h = [1+||a||_{\infty}\log^{r^+}(e+1)]^{p^+ +1}$, we also have 
    \begin{equation*}
        \phi_{\infty}(\beta t)\leq t^{p(x)}\leq \Phi(x, t) + h(x).
    \end{equation*}
    which proves that $\Phi(x, t)$ satisfies $(A2).$
   \end{proof}   
    \noindent  In the above lemma, we have assumed that the domain is bounded. For unbounded domains, the condition $(A2)$ follows if we assume an additional condition on $ p$, namely Nekvinda's decay condition:
     \begin{lemma}\label{(A2) theorem}
      Let $\Omega \subset \mathbb R^n, n\geq 2,$ be a domain. Assume that $p, q$ and $a$ are H\"older continuous functions while $r:\bar\Omega\rightarrow [0, \infty)$ is a log-log-H\"older continuous function with $1\leq p(x)\leq q(x)$ for all $x\in \bar \Omega$ and 
      \begin{equation*}
       \big(\frac{q}{p} \big)^{+}  < 1+ \frac{\gamma}{n},
      \end{equation*}          
     where $\gamma$ is the  H\"older exponent of $a(x).$ Assume also that $p$ satisfies the Nekvinda's decay condition. Then $\Phi(x,t):=t^{p(x)}+a(x)t^{q(x)}(\log(e+t))^{r(x)}$ satisfies $(A0), (A1), (A2), (aDec)$. 
    \end{lemma}
    \begin{proof}
        The proof of $(A0), (aDec)$ and $(A1)$ is exactly the same as in Theorem \ref{lemma for (A0)}, Lemma \ref{main lemma for increasing} and Proposition \ref{lemma for (A1)} since they are valid even if $\Omega$ is unbounded. Therefore, it remains to show the condition $(A2).$\\
        By Lemma \ref{equivalent lemma} (iii), it is equivalent to prove the condition $(A2)'.$ Towards this end, take $s=1,$ $\phi_{\infty}(t)= t^{p_{\infty}}$ and $\beta \leq 1.$ Note that $ \phi_{\infty}(t)\leq 1$ implies $t\leq 1$. We will consider two cases. For all the points $x$ where $p(x)< p_{\infty},$ we have by Young's inequality 
    \begin{eqnarray*}
        \Phi(x, \beta t)&\leq& [1+||a||_{\infty}\log^{r^+}(e+1)]\beta^{p(x)} t^{p(x)} \\
        &\leq& \frac{p(x)}{p_{\infty}}t^{p_{\infty}} + \frac{p_{\infty} -p(x)}{p_{\infty}}[1+||a||_{\infty}\log^{r^+}(e+1)]^{\frac{p_{\infty}}{p_{\infty} -p(x)}}\beta^{\frac{1}{|\frac{1}{p(x)}-\frac{1}{p_{\infty}}|}}\\
        &\leq& \phi_{\infty}(t)+ (\beta[1+||a||_{\infty}\log^{r^+}(e+1)])^{\frac{1}{|\frac{1}{p(x)}-\frac{1}{p_{\infty}}|}}.
    \end{eqnarray*}
    Let us take
    \begin{center}
    $\beta < c[1+||a||_{\infty}\log^{r^+}(e+1)]^{-1},$\quad \text{and}\\
     $h(x) = (\beta[1+||a||_{\infty}\log^{r^+}(e+1)])^{\frac{1}{|\frac{1}{p(x)}-\frac{1}{p_{\infty}}|}}$,
     \end{center}
     where $c\in(0, 1)$ is the constant from the Nekvinda's condition. Then we have $h\in L^1(\Omega)\cap L^{\infty}(\Omega)$. For all those $x$ where $p(x)\geq p_{\infty},$ we have, after taking the same $\beta,$  
     \begin{equation*}
         \Phi(x, \beta t) \leq [1+||a||_{\infty}\log^{r^+}(e+1)]\beta t^{p_{\infty}} \leq \phi_{\infty}(t).
     \end{equation*}
     We prove the other inequality similarly: when $p(x)\leq  p_{\infty},$ we have
    \begin{equation*}
        \phi_{\infty}(\beta t)\leq t^{p(x)}\leq \Phi(x, t),
    \end{equation*}
    and when $p(x)>p_{\infty}$, we obtain using Young's inequality
    \begin{equation*}
         \phi_{\infty}(\beta t)\leq \frac{p_{\infty}}{p(x)}t^{p(x)} + \frac{{p(x)}-p_{\infty} }{{p(x)}}\beta^{\frac{1}{|\frac{1}{p(x)}-\frac{1}{p_{\infty}}|}}\leq \Phi(x, t) + h(x),
    \end{equation*}
    which proves that $\Phi(x, t)$ satisfies $(A2).$
    \end{proof}
    Now we will prove the sufficient part of the Logarithmic Variable Exponent Double Phase Embedding using the following theorem.
     \begin{theorem}\cite[Corollary 6.3.3]{HH19}\label{embedding theorem}
      Let $\Omega$  be a John domain. Assume that $\Phi:\Omega\times[0,\infty)\rightarrow[0,\infty]$ is a generalized weak $\Phi$-function that satisfies $(A0), (A1), (A2), (aInc)$ and $(aDec)_q,$ $ q<n.$ Suppose that $\Psi\in \Phi_w(\Omega) $ satisfies $t^{\frac{-1}{n}} \Phi^{-1}(x, t) \approx \Psi^{-1}(x, t).$ Then $W^{1,\Phi(\cdot,\cdot)}(\Omega) \hookrightarrow L^{\Psi(\cdot,\cdot)}(\Omega) .$
    \end{theorem}
\vspace{0.05cm}    
\noindent\textbf{Proof of Theorem \ref{main theorem 1}}
Note that under the assumptions in (i), $\Phi(\cdot, \cdot)$ satisfies $(A0), (A1), (A2),$ $ (Inc)_{p^-}, (Dec)_{q^++r^+}$ using the Lemmas \ref{lemma for (A0)}, \ref{lemma for (A1)}, \ref{lemma for (A2)} and \ref{main lemma for increasing} respectively. Note that 
\begin{eqnarray*}
& &
\Psi^{-1}(x, t)\\ 
&\approx&  t^{-1/n}\Phi^{-1}(x,t) \\
& \approx& 
\min
\left\{
t^{1/p^*(x)},
\,
a(x)^{-1/n}
\left(a(x)^{-1}t\right)^{1/q^*(x)}
\left(
\log\left(e+a(x)^{-1}t\right)
\right)^{-r(x)/q(x)}
\right\},
\end{eqnarray*}
and hence $$\Psi(x,t)\approx t^{p^*(x)} + (a(x)^{\frac{1}{q(x)}}t)^{q^*(x)}\left(\log(e+ a(x)^{\frac{1}{n}}t)\right)^{\frac{r(x)q^*(x)}{q(x)}},$$
and therefore the embedding follows from Theorem \ref{embedding theorem}. For the second case, we use Lemma \ref{weaker assumption on p} instead of Lemma \ref{lemma for (A1)} to get the $(A1)$ property of $\Phi(\cdot, \cdot)$ and hence the embedding follows from Theorem \ref{embedding theorem}. \qed
\section{Proof of Theorem \ref{main theorem 3}}
\noindent In this section, we have shown that if the embedding holds for $W^{1,\Phi(\cdot, \cdot)}$, where $\Phi(x,t):=t^{p(x)}+ a(x) t^{q(x)}(\log(e+t))^{r(x)}$ with $0<a_0\leq a(x)\leq a_1<\infty,$ then $\Omega$ satisfies log-measure density condition. Note that, without loss of generality, we may assume that $0<a_0<1.$\\
\noindent  We first prove three lemmas to estimate the norm of the characteristic function of a measurable set. Here and hereafter, for a measurable set $A\subset\Omega$ and an exponent $q:\Omega\rightarrow [1,\infty),$ we denote $q_A^+:=\sup_{x\in A}q(x) $ and $q_A^-:=\inf_{x\in A}q(x).$
\begin{lemma}
Let $\Phi:\Omega\times[0,\infty)\rightarrow[0,\infty]$  be given by $\Phi(x,t):=t^{p(x)}+a(x)t^{q(x)}(\log(e+t))^{r(x)}$ with $r(x)\geq 0$ for all $x$ and let $A\subset\Omega$ be a measurable set. Then 
\begin{equation}\label{Norm 4 lemma proof equation}
     \min \{a_0^\frac{1}{q_{A}^+}|A|^\frac{1}{q_{A}^+},a_0^\frac{1}{q_{A}^-}|A|^\frac{1}{q_{A}^-} \} \leq \|1_{A}\|_{L^{\Phi(\cdot,\cdot)}(\Omega)} \leq 2(1+||a||_{\infty})^{\frac{1}{p_{A}^-}}\max \{ |A|^\frac{1}{q_{A}^+} \log^{r_{A}^+} (e+\frac{1}{|A|}),|A|^\frac{1}{q_{A}^-} \log^{r_{A}^+}(1+e) \}.
\end{equation}  
\end{lemma}

\begin{proof}
We start with the proof of the second inequality of \eqref{Norm 4 lemma proof equation}. Let $u>|A|$ and assume first that $u\leq 1$. Then, for any $x\in A,$ using $1\leq q_A^-\leq q(x) \leq q_A^+$ and $a(x)\leq 1+||a||_{\infty},$ we obtain
\begin{eqnarray*}
\Phi\bigg(x,\frac{1}{2(1+||a||_{\infty})^{\frac{1}{p_{A}^-}}u^{\frac{1}{q_{A}^+}}(\log(e+\frac{1}{u}))^{r_{A}^+}}\bigg)&=&\frac{1}{2^{p(x)}(1+||a||_{\infty})^{\frac{p(x)}{p_{A}^-}}u^{\frac{p(x)}{q_{A}^+}}(\log(e+\frac{1}{u}))^{p(x)r_{A}^+}} +\\
 & & \frac{a(x)\bigg(\log\Big(e+\frac{1}{2(1+||a||_{\infty})^{\frac{1}{p_{A}^-}}u^{\frac{1}{q _{A}^+}}\big(\log(e+\frac{1}{u})\big)^{r_{A}^+}}\Big)\bigg)^{r(x)}}{2^{q(x)}(1+||a||_{\infty})^{\frac{q(x)}{p_{A}^-}}u^{\frac{q(x)}{q_{A}^+}}(\log(e+\frac{1}{u}))^{q(x)r_{A}^+}}\\
 &\leq & \frac{1}{2u(\log(e+\frac{1}{u}))^{r_{A}^+}} +\\
& &\frac{\bigg(\log\Big(e+\frac{1}{2(1+||a||_{\infty})^{\frac{1}{p_{A}^-}}u^{\frac{1}{q_{A}^+}}(\log(e+\frac{1}{u}))^{r_{A}^+}}\Big)\bigg)^{r_{A}^+}}{2u(\log(e+\frac{1}{u}))^{r_{A}^+}}.
\end{eqnarray*}
Integrating the above inequality over $A$ and using $2(1+||a||_{\infty})^{\frac{1}{p_{A}^-}}u^{{\frac{1}{q_{A}^+}}-1}(\log(e+\frac{1}{u}))^{r_{A}^+} \geq 1,$ we get
\begin{eqnarray*}
    \int_A \Phi\bigg(x,\frac{1}{2(1+||a||_{\infty})^{\frac{1}{p_{A}^-}}u^{\frac{1}{q_{A}^+}}(\log(e+\frac{1}{u}))^{r_{A}^+}}\bigg)dx \leq \frac{|A|}{2u}+\frac{|A|}{2u}<1,
\end{eqnarray*}
and hence we have $\|1_{A}\|_{L^{\Phi(\cdot,\cdot)}(\Omega)} \leq 2(1+||a||_{\infty})^{\frac{1}{p_{A}^-}}u^\frac{1}{q_{A}^+} (\log (e+\frac{1}{u}))^{r_{A}^+}.$

If $u>1$, we can similarly show that $\|1_{A}\|_{L^{\Phi(\cdot,\cdot)}(\Omega)} \leq 2(1+||a||_{\infty})^{\frac{1}{p_{A}^-}}u^\frac{1}{q_{A}^-} (\log(1+e))^{r_{A}^+}$. The second inequality follows as $u \rightarrow |A|^+$.\\

Let us then prove the first inequality of \eqref{Norm 4 lemma proof equation}. Let $u<|A|$ and assume first that $u\leq 1$. Then 
\begin{equation*}
 \int_A \Phi\Big(x,\frac{1}{(a_0 u)^{\frac{1}{q_{A}^-}}}\Big)dx = \int_A \frac{1}{(a_0 u)^{\frac{p(x)}{q_{A}^-}}}+\frac{a(x)\Big(\log\big(e+\frac{1}{(a_0 u)^{\frac{1}{q_{A}^-}}}\big)\Big)^{r(x)}}{(a_0 u)^{\frac{q(x)}{q_{A}^-}}} dx \geq \frac{|A|}{u} >1.
\end{equation*}
 Hence, we get $(a_0 u)^\frac{1}{q_{A}^-}  \leq \|1_{A}\|_{L^{\Phi(\cdot,\cdot)}(\Omega)}$. If $u>1$, we can similarly show that $(a_0u)^\frac{1}{q_{A}^+}  \leq \|1_{A}\|_{L^{\Phi(\cdot,\cdot)}(\Omega)}$. The first inequality follows as $u \rightarrow |A|^-$.
\end{proof}

\begin{lemma}\label{Norm 5 lemma}
Let $\Phi:\Omega\times[0,\infty)\rightarrow[0,\infty]$  be given by $\Phi(x,t):=t^{p(x)}+a(x)t^{q(x)}(\log(e+t))^{r(x)}$ with $r(x) < 0$ for all $x$ and let $A\subset \Omega$ be a measurable set. Then 
\begin{eqnarray}\label{Norm 5}
& &
\min \{a_0|A|^{\frac{1}{q_{A}^-}}(\log(e+\frac{1}{|A|}))^{-r_{A}^-},a_0|A|^{\frac{1}{q_{A}^+}}(\log(e+\frac{1}{|A|}))^{-r_{A}^-} \} \nonumber \\ 
&\leq& \|1_{A}\|_{L^{\Phi(\cdot,\cdot)}(\Omega)} \\ \nonumber
&\leq & \max \{2(1+||a||_{\infty})^{\frac{1}{p_{A}^-}}|A|^\frac{1}{q_{A}^+},2(1+||a||_{\infty})^{\frac{1}{p_{A}^-}}|A|^\frac{1}{q_{A}^-} \}.
\end{eqnarray}
\end{lemma}
\begin{proof}
We start with the proof of the second inequality of \eqref{Norm 5}. Let $u>|A|$  and assume that $u\leq 1$. Then we have
\begin{eqnarray*}
    \int_A \Phi\big(x,\frac{1}{2(1+||a||_{\infty})^{\frac{1}{p_{A}^-}}u^{\frac{1}{q_{A}^+}}}\big)dx &= &\int_A \frac{1}{2^{p(x)}(1+||a||_{\infty})^{\frac{p(x)}{p_{A}^-}}u^{\frac{p(x)}{q_{A}^+}}}dx\\ 
    &+&\frac{a(x)\big(\log(e+\frac{1}{2(1+||a||_{\infty})^{\frac{1}{p_{A}^-}}u^{\frac{1}{q_{A}^+}}})\big)^{r(x)}}{2^{q(x)}(1+||a||_{\infty})^{\frac{q(x)}{p_{A}^-}}u^{\frac{q(x)}{q_{A}^+}}} dx \\
    &\leq&\frac{|A|}{2u}+\frac{|A|}{2u}\\  
    &<&1
\end{eqnarray*}
and hence $2(1+||a||_{\infty})^{\frac{1}{p_{A}^-}}u^\frac{1}{q_{A}^+}  \geq \|1_{A}\|_{L^{\Phi(\cdot,\cdot)}(\Omega)}$.  If $u>1$, we can similarly show that $2(1+||a||_{\infty})^{\frac{1}{p_{A}^-}}u^\frac{1}{q_{A}^-}  \geq \|1_{A}\|_{L^{\Phi(\cdot,\cdot)}(\Omega)}$. The second inequality follows as $u \rightarrow |A|^+$.\\

Let us then prove the first inequality of \eqref{Norm 5}. Let $u<|A|$, and assume first that $u\leq 1$. Then 
\begin{eqnarray*}
 \int_A \Phi\Big(x,\frac{1}{a_0u^{\frac{1}{q_{A}^-}}(\log(e+\frac{1}{u}))^{-r_{A}^-}}\Big)dx &\geq& \int_A \frac{1}{(a_0)^{p(x)}(\log(e+\frac{1}{u}))^{-p(x)r_{A}^-}u^{\frac{p(x)}{q_{A}^-}}} dx \\
 &+& \int_A \frac{a(x) \Big(\log\big(e+\frac{1}{(\log(e+\frac{1}{u}))^{-r_{A}^-}u^{\frac{1}{q_{A}^-}}}\big)\Big)^{r_{A}^-}}{a_0(\log(e+\frac{1}{u}))^{-q_{A}^+ r_{A}^-} u^{\frac{q(x)}{q_{A}^-}}} dx\\
 &\geq & \frac{|A|}{u} >1.
\end{eqnarray*}
 Hence, we get $a_0u^{\frac{1}{q_{A}^-}}(\log(e+\frac{1}{u}))^{-r_{A}^-} \leq \|1_{A}\|_{L^{\Phi(\cdot,\cdot)}(\Omega)}$. If $u>1$, we can similarly show that $a_0(\log(e+\frac{1}{u}))^{-r_{A}^-}u^\frac{1}{q_{A}^+}  \leq \|1_{A}\|_{L^{\Phi(\cdot,\cdot)}(\Omega)}$. The first inequality follows as $u \rightarrow |A|^-$.
\end{proof}
\begin{lemma}\label{Norm 6 lemma}
Let $\Phi:\Omega\times[0,\infty)\rightarrow[0,\infty)$  be given by $\Phi(x,t):=t^{p(x)}+a(x)t^{q(x)}(\log(e+t))^{r(x)}$ where $r(x) < 0$ for some $x $ and $r(x)\geq 0$ for some $x$. Assume $A\subset \Omega$ is a measurable set with $|A|< \frac{1}{2}$. Then there exists constant $b_2>0$ such that
\begin{equation}\label{Norm 6}
     a_0^\frac{1}{q_{A}^-} |A|^\frac{1}{q_{A}^-}  \leq \|1_{A}\|_{L^{\Phi(\cdot,\cdot)}(\Omega)} \leq b_2 \max  4(1+||a||_{\infty})^{\frac{1}{p_{A}^-}} \{ |A|^\frac{1}{q_{A}^+} \big(\log (e+\frac{1}{|A|})\big)^{r_{A}^+} ,|A|^\frac{1}{q_{A}^-} \big(\log(1+e)\big)^{r_{A}^+} \}.
\end{equation}
\end{lemma}
\begin{proof}
We start with the proof of the second inequality of \eqref{Norm 6}. Let $u>|A|$ and assume first that $u\leq 1$. Here, we write $A=A_1\cup A_2,$ where $A_1=\{x\in A: r(x)\geq 0\}$ and $A_2=\{x\in A: r(x)< 0\}.$ For any $x\in A_1,$ we have
\begin{eqnarray*}
    \Phi\Big(x,\frac{1}{4(1+||a||_{\infty})^{\frac{1}{p_{A}^-}}u^{\frac{1}{q_{A}^+}}(\log\big(e+\frac{1}{u}\big)\big)^{r_{A}^+}}\Big) &=& \frac{1}{4^{p(x)}(1+||a||_{\infty})^{\frac{p(x)}{p_{A}^-}}u^{\frac{p(x)}{q_{A}^+}}(\log(e+\frac{1}{u}))^{p(x)r_{A}^+}}+\\
 & & \frac{a(x)\bigg(\log\Big(e+\frac{1}{4(1+||a||_{\infty})^{\frac{1}{p_{A}^-}}u^{\frac{1}{q_{A}^+}}\big(\log(e+\frac{1}{u})\big)^{r_{A}^+}}\Big)\bigg)^{r(x)}}{4^{q(x)}(1+||a||_{\infty})^{\frac{q(x)}{p_{A}^-}}u^{\frac{q(x)}{q_{A}^+}}(\log(e+\frac{1}{u}))^{q(x)r_{A}^+}}\\
 &\leq & \frac{1}{4u(\log(e+\frac{1}{u}))^{r_A^+}}+\\
 & & \frac{\bigg(\log\Big(e+\frac{1}{4(1+||a||_{\infty})^{\frac{1}{p_{A}^-}}u^{\frac{1}{q_{A}^+}}(\log(e+\frac{1}{u}))^{r_{A}^+}}\Big)\bigg)^{r_{A}^+}}{4u(\log(e+\frac{1}{u}))^{r_{A}^+}},
\end{eqnarray*}
hence after integrating over $A_1$ and using $4(1+||a||_{\infty})^{\frac{1}{p_{A}^-}}u^{{\frac{1}{q_{A}^+}}-1}(\log(e+\frac{1}{u}))^{r_{A}^+} \geq 1,$ we have
\begin{equation*}
    \int_{A_1} \Phi\Big(x,\frac{1}{4(1+||a||_{\infty})^{\frac{1}{p_{A}^-}}u^{\frac{1}{q_{A}^+}}(\log\big(e+\frac{1}{u}\big)\big)^{r_{A}^+}}\Big)dx\leq \frac{|A|}{4u}+\frac{|A|}{4u}=\frac{|A|}{2u}.
\end{equation*}
For any $x\in A_2,$ we easily have, using $1\leq q_A^- \leq q(x) \leq q_A^+,$ $u\leq 1,$ $a(x)\leq 1+||a||_{\infty}$ and $r(x)<0,$
\begin{equation*}
    \int_{A_2}\Phi\Big(x,\frac{1}{4(1+||a||_{\infty})^{\frac{1}{p_{A}^-}}u^{\frac{1}{q_{A}^+}}(\log\big(e+\frac{1}{u}\big)\big)^{r_{A}^+}}\Big) dx\leq \frac{|A|}{4u}+\frac{|A|}{4u}=\frac{|A|}{2u}.
\end{equation*}
Since $u>|A|,$ we finally have
\begin{equation*}
    \int_{A}\Phi\Big(x,\frac{1}{4(1+||a||_{\infty})^{\frac{1}{p_{A}^-}}u^{\frac{1}{q_{A}^+}}(\log\big(e+\frac{1}{u}\big)\big)^{r_{A}^+}}\Big) dx <1
\end{equation*}
and hence
we have $\|1_{A}\|_{L^{\Phi(\cdot,\cdot)}(\Omega)} \leq 4(1+||a||_{\infty})^{\frac{1}{p_{A}^-}}u^\frac{1}{q_{A}^+} (\log (e+\frac{1}{2u}))^{r_{A}^+}$.
If $u>1$, we can similarly show that $\|1_{A}\|_{L^{\Phi(\cdot,\cdot)}(\Omega)} \leq 4(1+||a||_{\infty})^{\frac{1}{p_{A}^-}}u^\frac{1}{q_{A}^-} (\log(1+e))^{r_{A}^+}$. The second inequality follows as $u \rightarrow |A|^+$.\\
Let us then prove the first inequality of \eqref{Norm 6}. Let $u<|A|$. Then 
\begin{eqnarray*}
\int_A \Phi\Big(x,\frac{1}{(a_0 u)^{\frac{1}{q_{A}^-}}}\Big)dx &=& \int_A \frac{1}{(a_0 u)^{\frac{p(x)}{q_{A}^-}}}+\frac{a(x)\Big(\log\big(e+\frac{1}{(a_0 u)^{\frac{1}{q_{A}^-}}}\big)\Big)^{r(x)}}{(a_0 u)^{\frac{q(x)}{q_{A}^-}}} dx\\
&=& \int_{A_1} \frac{1}{(a_0u)^{\frac{p(x)}{q_{A}^-}}}+\frac{a(x)\Big(\log\big(e+\frac{1}{(a_0u)^{\frac{1}{q_{A}^-}}}\big)\Big)^{r(x)}}{(a_0u)^{\frac{q(x)}{q_{A}^-}}} dx +\\
& &\int_{A_2} \frac{1}{(a_0u)^{\frac{p(x)}{q_{A}^-}}}+\frac{a(x)\Big(\log\big(e+\frac{1}{(a_0u)^{\frac{1}{q_{A}^-}}}\big)\Big)^{r(x)}}{(a_0u)^{\frac{q(x)}{q_{A}^-}}} dx  \\
&\geq& \frac{|A|}{u} > 1.
\end{eqnarray*}
Hence we get  $(a_0u)^\frac{1}{q_{A}^-}   \leq \|1_{A}\|_{L^{\Phi(\cdot,\cdot)}(\Omega)}.$ The first inequality follows as $u \rightarrow |A|^-$.
\end{proof}
To complete the proof, we first need the following lemma:
  \begin{lemma} \label{lemma for theorem 1.3}
 If we have the same assumptions as in Theorem \ref{main theorem 3}, then there exist positive constants $C_1$, $C_2$, $C_3$ such that for all $x$ in $\bar{\Omega}$ and every $R$ in $]0,1]$ we have
\begin{equation}\label{ineuqality 13}
    R- \tilde{R} \leq C_1 \frac{(1+||a||_{\infty})^{\frac{1}{p^-}}}{a_0^{\frac{1}{\frac{1}{q^-} -\frac{1}{n}}}}2^{\frac{1}{p^-} -\frac{1}{n}+1}\ |A_R|^{\frac{1}{q_{A_R}^+} -\frac{1}{q_{A_R}^-} +\frac{1}{n}} (\log (e+\frac{1}{|A_R|}))^{r_{A_R}^+} 
 \end{equation}
when $r(x) \geq 0$ for all $x$;
\begin{equation}\label{inequality 14}
    R- \tilde{R} \leq C_2 \frac{(1+||a||_{\infty})^{\frac{1}{p^-}}}{a_0}2^{\frac{1}{p^-} -\frac{1}{n}+1}\ |A_R|^{\frac{1}{q_{A_R}^+} -\frac{1}{q_{A_R}^-} +\frac{1}{n}} (\log (e+\frac{1}{|A_R|}))^{r_{A_R}^-}
\end{equation}
when $r(x) < 0$ for all $x$; and 
\begin{equation}\label{inequlity 15}
    R- \tilde{R} \leq  C_3\frac{(1+||a||_{\infty})^{\frac{1}{p^-}}}{a_0^{\frac{1}{\frac{1}{q^-} -\frac{1}{n}}}} 2^{\frac{1}{p^-} -\frac{1}{n}}\ |A_R|^{\frac{1}{q_{A_R}^+} -\frac{1}{q_{A_R}^-} +\frac{1}{n}} (\log (e+\frac{1}{|A_R|}))^{r_{A_R}^+} 
\end{equation}
when  $r(x) < 0$ for some $x$ and $r(x)\geq 0$ for some $x$. 
\end{lemma}       
 \begin{proof}    
Since by Corollary 6.3.3 of \cite{HH19} $W^{1,\Phi(\cdot,\cdot)}(\Omega) \hookrightarrow L^{\Psi(\cdot,\cdot)}(\Omega),$ there exists a constant $c_1 >0$ such that whenever $u\in W^{1,\Phi(\cdot,\cdot)} (\Omega)$ one has the inequality
 \begin{equation} \label{inequality 16}
       ||u||_{L^{\Psi(\cdot,\cdot)}(\Omega)}\leq c_1|| u||_{W^{1,\Phi(\cdot,\cdot)}(\Omega)}.
 \end{equation}
For a fixed  $x\in\bar{\Omega}$ let $u(y):=\phi(y-x),$ where $y\in\Omega$ and $\phi$ is a cut-off function so that
\begin{enumerate}
\item $\phi:\mathbb{R}^n \rightarrow [0,1]$, 
\item $\hbox{spt}\ \phi \subset B_R(0)$,
\item  $\phi|_{B_{\tilde{R}}(0)}=1$, and 
\item $|\nabla \phi | \leq \tilde{c} / (R-\tilde{R})$ for some constant $\tilde{c}$.
\end{enumerate}
Note that we have inequalities
\[
\|1_{B_{\tilde{R}}}\|_{L^{\Psi(\cdot,\cdot)(\Omega)}}\leq \|u\|_{L^{\Psi(\cdot,\cdot)(\Omega)}}, \quad\quad \|u\|_{L^{\Phi(\cdot,\cdot)(\Omega)}}\leq \|1_{B_{R}}\|_{L^{\Phi(\cdot,\cdot)(\Omega)}}
\]
and 
\[
    \|\nabla u\|_{L^{\Phi(\cdot,\cdot)(\Omega)}} \leq \frac{\tilde{c}}{R-\tilde{R}}\|1_{B_R \thicksim B_{\tilde{R}}}\|_{L^{\Phi(\cdot,\cdot)(\Omega)}}
 \leq  \frac{\tilde{c}}{R-\tilde{R}}\ \| 1_{B_R}\|_{L^{\Phi(\cdot,\cdot)(\Omega)}} .
\]
Use these inequalities in inequality \eqref{inequality 16} to obtain
\begin{eqnarray*} 
\|1_{B_{\tilde{R}}}\|_{L^{\Psi(\cdot,\cdot)(\Omega)}} &\leq & c_1(\|1_{B_{R}}\|_{L^{\Phi(\cdot,\cdot)(\Omega)}}+\frac{\tilde{c}}{R-\tilde{R}}\ \| 1_{B_R}\|_{L^{\Phi(\cdot,\cdot)(\Omega)}})\\
&\leq & \frac{2c_1\max\{1,\tilde{c}\}}{R-\tilde{R}}\|1_{B_{R}}\|_{L^{\Phi(\cdot,\cdot)(\Omega)}}
\end{eqnarray*}
and hence
\begin{eqnarray*}
 R- \tilde{R} \leq {c_2} \frac{\|1_{B_R}\|_{L^{\Phi(\cdot,\cdot)(\Omega)}}}{ \| 1_{B_{\tilde{R}}}\|_{L^{\Psi(\cdot,\cdot)(\Omega)}}},
\end{eqnarray*}
where $c_2:=2c_1\max\{1,\tilde{c}\}.$\\ 

Case-1 ($r(x) \geq 0 $ for all $x\in \Omega$): Using the norm estimates for $\Phi$ in Lemma \ref{Norm 4 lemma proof equation}, and the similar estimates for $\Psi,$ we get
 \begin{eqnarray*}
 R- \tilde{R} &\leq& c_2\ 2(1+||a||_{\infty})^{\frac{1}{p_{A_R}^-}}\frac{|A_R|^\frac{1}{q_{A_R}^+} (\log (e+\frac{1}{|A_R|}))^{r_{A_R}^+}} {a_0^\frac{1}{q*_{A_R^-}}|A_{\tilde{R}}|^\frac{1}{q*_{A_R^-}}} \\
              &=& c_2\ 2(1+||a||_{\infty})^{\frac{1}{p_{A_R}^-}}\frac{|A_R|^\frac{1}{q_{A_R}^+} (\log (e+\frac{1}{|A_R|}))^{r_{A_R}^+}} {a_0^{\frac{1}{q_{A_R}^-} -\frac{1}{n}}|A_{\tilde R}|^{\frac{1}{q_{A_R}^-} -\frac{1}{n}}}\\
              &=&  c_2\ \frac{(1+||a||_{\infty})^{\frac{1}{p_{A_R}^-}}}{a_0^{\frac{1}{\frac{1}{q_{A_R}^-} -\frac{1}{n}}}}2^{\frac{1}{p_{A_R}^-} -\frac{1}{n}+1}|A_R|^{\frac{1}{q_{A_R}^+} -\frac{1}{q_{A_R}^-} +\frac{1}{n}} (\log (e+\frac{1}{|A_R|}))^{r_{A_R}^+}\\
              &\leq & c_2 \frac{(1+||a||_{\infty})^{\frac{1}{p^-}}}{a_0^{\frac{1}{\frac{1}{q^-} -\frac{1}{n}}}} 2^{\frac{1}{p^-} -\frac{1}{n}+1}\ |A_R|^{\frac{1}{q_{A_R}^+} -\frac{1}{q_{A_R}^-} +\frac{1}{n}} (\log (e+\frac{1}{|A_R|}))^{r_{A_R}^+}, 
\end{eqnarray*}
as claimed.\\
Case-2 ($r(x) \leq 0$ for all $x\in \Omega$): Using the norm estimates for $\Phi$ in Lemma \ref{Norm 5 lemma} and the similar estimates for $\Psi,$ we get 
\begin{eqnarray*}
 R- \tilde{R} &\leq& c_2\ 2(1+||a||_{\infty})^{\frac{1}{p_{A_R}^-}}\frac{|A_R|^\frac{1}{q_{A_R}^+} (\log (e+\frac{1}{|A_R|}))^{r_{A_R}^-}} {a_0|A_{\tilde{R}}|^\frac{1}{q*_{A_R}^-}}  \\
              &=& c_2\ 2(1+||a||_{\infty})^{\frac{1}{p_{A_R}^-}}\frac{|A_R|^\frac{1}{q_{A_R}^+}(\log (e+\frac{1}{|A_R|}))^{r_{A_R}^-}} {a_0|A_{\tilde R}|^{\frac{1}{q_{A_R}^-} -\frac{1}{n}}}\\
              &\equiv &  c_2 \frac{(1+||a||_{\infty})^{\frac{1}{p_{A_R}^-}}}{a_0}2^{\frac{1}{p_{A_R}^-} -\frac{1}{n}+1}|A_R|^{\frac{1}{q_{A_R}^+} -\frac{1}{q_{A_R}^-} +\frac{1}{n}} (\log (e+\frac{1}{|A_R|}))^{r_{A_R}^-}\\
              &\leq & c_2 \frac{(1+||a||_{\infty})^{\frac{1}{p^-}}}{a_0}2^{\frac{1}{p^-} -\frac{1}{n}+1}\ |A_R|^{\frac{1}{q_{A_R}^+} -\frac{1}{q_{A_R}^-} +\frac{1}{n}} (\log (e+\frac{1}{|A_R|}))^{r_{A_R}^-}.
\end{eqnarray*}
Case-3 ($r(x) < 0$ for some $x\in \Omega$ and $r(x)\geq 0$ for some $x\in \Omega$): Using the norm estimates for $\Phi$ in Lemma \ref{Norm 6 lemma} and the similar estimates for $\Psi,$ we get 
\begin{eqnarray*}
 R- \tilde{R} &\leq& c_2\ 4(1+||a||_{\infty})^{\frac{1}{p_{A_R}^-}}\frac{|A_R|^\frac{1}{q_{A_R}^+} (\log (e+\frac{1}{|A_R|}))^{r_{A_R}^+}} {a_0^\frac{1}{q*_{A_R^-}}|A_{\tilde{R}}|^\frac{1}{q*_{A_R^-}}} \\
              &=& c_2\ 4(1+||a||_{\infty})^{\frac{1}{p_{A_R}^-}}\frac{|A_R|^\frac{1}{q_{A_R}^+} (\log (e+\frac{1}{|A_R|}))^{r_{A_R}^+}} {a_0^{\frac{1}{q_{A_R}^-} -\frac{1}{n}}|A_{\tilde R}|^{\frac{1}{q_{A_R}^-} -\frac{1}{n}}}\\
              &=&  c_3\ \frac{(1+||a||_{\infty})^{\frac{1}{p_{A_R}^-}}}{a_0^{\frac{1}{\frac{1}{q_{A_R}^-} -\frac{1}{n}}}}2^{\frac{1}{p_{A_R}^-} -\frac{1}{n}}|A_R|^{\frac{1}{q_{A_R}^+} -\frac{1}{q_{A_R}^-} +\frac{1}{n}} (\log (e+\frac{1}{|A_R|}))^{r_{A_R}^+}\\
              &\leq & c_3\frac{(1+||a||_{\infty})^{\frac{1}{p^-}}}{a_0^{\frac{1}{\frac{1}{q^-} -\frac{1}{n}}}} 2^{\frac{1}{p^-} -\frac{1}{n}}\ |A_R|^{\frac{1}{q_{A_R}^+} -\frac{1}{q_{A_R}^-} +\frac{1}{n}} (\log (e+\frac{1}{|A_R|}))^{r_{A_R}^+},
\end{eqnarray*}
where $c_3= 4c_2$.
\end{proof}
\textbf{Proof of Theorem \ref{main theorem 3}}
For a fixed $x$ in $\bar\Omega$ define $A_R:= B_R(x)\cap \Omega$. It is enough to consider the case when $|A_R| \leq 1$, otherwise $|A_R|\geq 1 \geq \mathbb R^n$ whenever $R\leq 1$ and there is nothing to prove. Moreover, it is enough to consider $R\leq r_0$ for some $0<r_0\leq 1/4.$ For such an $R$, denote by $\tilde{R}\leq R$ the smallest real number such that $|A_{\Tilde{R}}|= \frac{1}{2} |A_R|$. \\
Now, construct the sequence $\{  R_i \}$ by setting $R_0:=R$, and then define $R_{i+1}:=\tilde{R_i}$ inductively for $i \geq 0$.
It follows that
\[
     |A_{R_{i}}|=\frac{1}{2^i}\ | A_R|, 
\]
with $\lim_{i \rightarrow \infty} R_i=0$.\\

Case-1 ($r(x) \geq 0 $ for all $x\in \Omega$): Using the inequality \eqref{ineuqality 13} one obtains
\begin{eqnarray*}     
  R_i -R_{i+1} &\leq& C_1\frac{(1+||a||_{\infty})^{\frac{1}{p^-}}}{a_0^{\frac{1}{\frac{1}{q^-} -\frac{1}{n}}}} 2^{\frac{1}{p^-} -\frac{1}{n}+1} |A_{R_i}|^{\frac{1}{n}+\frac{1}{q_{A_{R_i}}^+}-\frac{1}{q_{A_{R_i}}^-}}(\log (e+\frac{1}{|A_{R_i}|}))^{r_{A_{R_i}}^+}\\ 
	      &\leq&  C_1 \frac{(1+||a||_{\infty})^{\frac{1}{p^-}}}{a_0^{\frac{1}{\frac{1}{q^-} -\frac{1}{n}}}} 2^{\frac{1}{p^-} -\frac{1}{n}+1}|A_{R_i}|^{\frac{1}{n}+\frac{1}{q_{A_R}^+}-\frac{1}{q_{A_R}^-}}(\log (e+\frac{1}{|A_{R_i}|}))^{r_{A_R}^+}\\
             &=& C_1 \frac{(1+||a||_{\infty})^{\frac{1}{p^-}}}{a_0^{\frac{1}{\frac{1}{q^-} -\frac{1}{n}}}}2^{\frac{1}{p^-} -\frac{1}{n}+1}\frac{|A_R|^{\eta_R}}{2^{i \eta_R}}(\log (e+\frac{2^i}{|A_{R}|}))^{r_{A_R}^+},
\end{eqnarray*}
where $\eta_R :=\frac{1}{n}+\frac{1}{q_{A_R}^+}-\frac{1}{q_{A_R}^-}.$  
Since we have $(\log (e+\frac{2^i}{|A_{R}|}))^{r_{A_R}^+}\leq i^{r_{A_R}^+}(\log (e+\frac{2}{|A_{R}|}))^{r_{A_R}^+}\leq 2^{r^+}i^{r_{A_R}^+}(\log (e+\frac{1}{|A_{R}|}))^{r_{A_R}^+}$ for $i\geq 1,$ we obtain 
\begin{eqnarray*}     
  R_i -R_{i+1} 
             &\leq & c_3\frac{i^{r_{A_R}^+}|A_R|^{\eta_R}}{2^{i \eta_R}}(\log (e+\frac{1}{|A_{R}|}))^{r_{A_R}^+},
\end{eqnarray*}
where $c_3= C_1\frac{(1+||a||_{\infty})^{\frac{1}{p^-}}}{a_0^{\frac{1}{\frac{1}{q^-} -\frac{1}{n}}}} 2^{\frac{1}{p^-} -\frac{1}{n}+1}2^{r^+}.$\\
Now, we would like to find a constant $\eta>0$, independent of $x$ and $R$, such that $\frac{1}{n}+\frac{1}{q_{A_R}^+}-\frac{1}{q_{A_R}^-} =: \eta_R\geq \eta>0$ for all $R\leq r_0.$ Toward this end, the log-H\"older continuity of $q$ gives, for any $z$ and $y$ in $A_R,$
\[
\vert \frac{1}{q(z)}-\frac{1}{q(y)}\vert\leq \frac{C_{\text{log}}}{\log(e+1/|z-y|)},
\]
taking the supremum over all pairs of points in $A_R$ one gets
\begin{equation}   \label{estimate1/p 1}
\frac{1}{q_{A_R}^-}-\frac{1}{q_{A_R}^+}\leq \frac{C_{\text{log}}}{\log(1/(2R))}.
\end{equation}
Suppose now that for some $R\leq 1/4$ we have $\eta_R\leq 0$. Then (\ref{estimate1/p 1}) gives
$ \frac{1}{n}\leq \frac{C_{\text{log}}}{\log\left(\frac{1}{2R}\right)}$,
which further implies
$R\geq \frac{1}{2}e^{-nC_{\text{log}}}$.

Hence, we have the following conclusion:
\begin{itemize} 
\item If $\frac{1}{2}e^{-nC_{\text{log}}}>\frac{1}{4}$, then there is no $R\leq \frac{1}{4}$ for which $\eta_R\leq 0$.
\item If $\frac{1}{2}e^{-nC_{\text{log}}}\leq \frac{1}{4}$, then $\eta_R\leq 0$ implies $R\geq \frac{1}{2}e^{-nC_{\text{log}}}.$ 
\end{itemize}

Therefore, if we choose $r_0=\frac{1}{2}\min\{\frac{1}{4},\frac{1}{2}e^{-nC_{\text{log}}}\}$, then $\eta_{r_0}>0$, and also
\begin{equation}\label{estimate of r_0 1} 
 \frac{1}{n}>\frac{C_{\text{log}}}{\log(1/(2r_0))}.
\end{equation} 
  
But $\eta_{r_0}$ may depend on the point $x$ fixed at the beginning of the proof. To obtain the required $\eta$, we apply again log-H\"older continuity of $1/q$ on $A_{r_0}$, to obtain
 \begin{equation}\label{estimate in A_r_0 1}
 \frac{1}{q_{A_{r_0}}^-}-\frac{1}{q_{A_{r_0}}^+}\leq \frac{C_{\text{log}}}{\log(1/(2r_0))},
 \end{equation}
 and \eqref{estimate of r_0 1} together with \eqref{estimate in A_r_0 1} give
 \[
 \eta_{r_0}=\frac{1}{n}+\frac{1}{q_{A_{r_0}}^+}-\frac{1}{q_{A_{r_0}}^-}\geq \frac{1}{n}-\frac{C_{\text{log}}}{\log(1/(2r_0))}>0.
 \]
 Choosing $\eta:=\frac{1}{n}-\frac{C_{\text{log}}}{\log(1/(2r_0))},$ we get that $\eta_R\geq \eta_{r_0}\geq \eta>0$ for all $R\leq r_0$. This is our desired
$\eta.$\\

Note that $\eta_R \geq \eta >0$ and using the integral test we have $\sum_{i=1}^{\infty} {i^{r_{A_R}^+}}{2^{-i\eta_R}}\leq \frac{(r_{A_R}^+)!}{(\eta_R \log2)^{(r_{A_R}^+ +1)}},$ where for any positive number $a,$ we have $a!= \Gamma(a+1),$ and hence
\begin{eqnarray*}     
 R = \sum_{i=0}^{\infty }(R_i -R_{i+1}) 
 &\leq& c_3 |A_R|^{\eta_R}(\log (e+\frac{1}{|A_{R}|}))^{r_{A_R}^+} \Big(1+\sum_{i=1}^{\infty} {i^{r_{A_R}^+}}{2^{-i\eta_R}}\Big)\\
  &\leq& c_3 |A_R|^{\eta_R}(\log (e+\frac{1}{|A_{R}|}))^{r_{A_R}^+}\Big(\frac{(r_{A_R}^+)!}{(\eta_R \log2)^{(r_{A_R}^+ +1)}}+1\Big)\\
   &\leq& c_3 |A_R|^{\eta_R}(\log (e+\frac{1}{|A_{R}|}))^{r_{A_R}^+}\Big(\frac{(r_{A_R}^+)!}{(\eta \log2)^{(r_{A_R}^+ +1)}}+1\Big)\\
   &\leq& c_3 |A_R|^{\eta_R}(\log (e+\frac{1}{|A_{R}|}))^{r_{A_R}^+}\Big(\frac{(r^+)!}{({\min(1,\eta)} \log2)^{(r^+ +1)}}+1\Big).
 \end{eqnarray*}
Moreover, since $ c_4:=1/ \max\Big\{1,\Big(\frac{c_3(r^+)!}{({\min(1,\eta)} \log2)^{(r^+ +1)}}+c_3\Big)\Big\} \leq 1,$ one has 
\begin{eqnarray}    \label{AR 1}
 |A_R| (\log (e+\frac{1}{|A_{R}|}))^\frac{r_{A_R}^+}{\eta_R}  \geq   c_4^{1 /\eta_R}  R^{1/\eta_R}
 \geq  c_4^{ 1 /\eta} R^{1 /\eta_R}
 = c_4^{1 / \eta} R^n R^{\beta_R / \eta_R} ,
\end{eqnarray}
where $\beta_R := 1-n \eta_R$.

Therefore, from equation \eqref{AR 1} one sees that if a positive lower bound for $R^{\beta_R / \eta_R}$ 
is provided, the proof of Theorem \ref{main theorem 1} will be completed. To achieve such a lower bound, we see that from the log-H\"{o}lder continuity of $q,$ 
\[
\vert q(z)-q(y)\vert\leq \frac{C_{\text{log}}}{\log(e+1/|z-y|)} ;
\]
taking the supremum over pairs of points in $A_R$, one gets
\[
q_{A_R}^+-q_{A_R}^-\leq \frac{C_{\text{log}}}{\log(1/(2R))} ,
\]
or
\[
\log\left(1 / (2R)^{q_{A_R}^+-q_{A_R}^-}\right)\leq C_{\text{log}},
\]
therefore,
\begin{equation} \label{inequality 4b}    
  R^{q_{A_R}^+-q_{A_R}^-}\geq \frac{e^{-C_{\text{log}}}}{2^{q_{A_R}^+- q_{A_R}^-  }} \geq \frac{e^{-C_{\text{log}}}}{2^{(q^+-q^-)}}.
\end{equation}
But
\[
R^{\frac{\beta_R}{\eta_R}} \geq R^{\frac{\beta_R}{\eta}}
= R^{\frac{n(q_{A_R}^+ - q_{A_R}^-)}{\eta q_{A_R}^+  q_{A_R}^-  }} \geq   (R^{q_{A_R}^+ - q_{A_R}^-})^{n / \eta (q^-)^2} ,
\]
hence using \eqref{inequality 4b} the required bound
\[  
R^{\frac{\beta_R}{\eta_R}} \geq  \left(\frac{e^{-C_{\text{log}}}}{2^{(q^+-q^-)}} \right)^{n / \eta (q^-)^2} =: c_5 >0.
\]
 Taking $f(t)=t (\log (e+\frac{1}{t}))^\frac{q_{A_R}^+}{\eta_R}$, we see that \eqref{AR 1} becomes $f(|A_R|)\geq cR^n,$ where $c := c_4^{1 / \eta} c_5$ and hence $|A_R|\geq f^{-1}(cR^n)$ which further implies that $c
 R^n (\log (e+\frac{1}{R}))^\frac{-r^+}{\eta} \leq c
 R^n (\log (e+\frac{1}{R}))^\frac{-r_{A_R}^+}{\eta_R} \leq | B_R(x) \cap \Omega | $. So, $\Omega$ satisfies the log-measure density condition.\\

Case-2 ($r(x) \leq 0$ for all $x\in \Omega$): Using the inequality \eqref{inequality 14} one obtains 
\begin{eqnarray*}
R_i -R_{i+1} &\leq& C_2 \frac{(1+||a||_{\infty})^{\frac{1}{p^-}}}{a_0}2^{\frac{1}{p^-} -\frac{1}{n}+1}|A_{R_i}|^{\frac{1}{n}+\frac{1}{q_{A_{R_i}}^+}-\frac{1}{q_{A_{R_i}}^-}}\\ 
	     &\leq&  C_2 \frac{(1+||a||_{\infty})^{\frac{1}{p^-}}}{a_0}2^{\frac{1}{p^-} -\frac{1}{n}+1}|A_{R_i}|^{\frac{1}{n}+\frac{1}{q_{A_R}^+}-\frac{1}{q_{A_R}^-}}\\
             &=& C_2 \frac{(1+||a||_{\infty})^{\frac{1}{p^-}}}{a_0}2^{\frac{1}{p^-} -\frac{1}{n}+1}\frac{|A_R|^{\eta_R}}{2^{i \eta_R}} \\
             &\leq & c_6{|A_R|^{\eta_R}}{2^{-i \eta_R}},
\end{eqnarray*}
where $\eta_R :=\frac{1}{n}+\frac{1}{q_{A_R}^+}-\frac{1}{q_{A_R}^-}$ and $c_6=\frac{C_2}{a_0}(1+||a||_{\infty})^{\frac{1}{p^-}}2^{\frac{1}{p^-} -\frac{1}{n}+1} .$ Note that $\eta_R \geq \eta >0 $ and hence
\begin{eqnarray*}     
 R = \sum_{i=0}^{\infty }(R_i -R_{i+1}) 
 &\leq& c_6 |A_R|^{\eta_R} \Big(\sum_{i=0}^{\infty} {2^{-i\eta_R}}\Big)\\
 &=& c_6 |A_R|^{\eta_R} \frac{1}{1-2^{-\eta_R}}\\
 &\leq&c_6 |A_R|^{\eta_R} \frac{1}{1-2^{-\eta}}\\
 &=& c_7 |A_R|^{\eta_R},
 \end{eqnarray*}
 where $c_7=  \frac{c_6}{1-2^{-\eta}}$.
 Moreover, since $ c_8:=1/ \max\{1,c_7\} \leq 1$ one has
 \begin{eqnarray}    \label{BR 1}
 |A_R|   \geq   c_8^{1 /\eta_R}  R^{1/\eta_R}
 \geq  c_8^{ 1 /\eta} R^{1 /\eta_R}
 = c_8^{1 / \eta} R^n R^{\beta_R / \eta_R} ,
\end{eqnarray}
where $\beta_R := 1-n \eta_R$.
Now we can proceed in a similar way as in Case-1 to obtain $cR^n  \leq  | B_R(x) \cap \Omega | $. So, $\Omega$ satisfies the measure density condition.\\
 
  Case-3 ($r(x) < 0$ for some $x\in \Omega$ and $r(x)\geq 0$ for some $x\in \Omega$): Using the inequality \eqref{inequlity 15}, one obtains 
  \begin{eqnarray*}     
  R_i -R_{i+1} &\leq& C_1\frac{(1+||a||_{\infty})^{\frac{1}{p^-}}}{a_0^{\frac{1}{\frac{1}{q^-} -\frac{1}{n}}}} 2^{\frac{1}{p^-} -\frac{1}{n}+1} |A_{R_i}|^{\frac{1}{n}+\frac{1}{q_{A_{R_i}}^+}-\frac{1}{q_{A_{R_i}}^-}}(\log (e+\frac{1}{|A_{R_i}|}))^{r_{A_{R_i}}^+}\\ 
	      &\leq&  C_1 \frac{(1+||a||_{\infty})^{\frac{1}{p^-}}}{a_0^{\frac{1}{\frac{1}{q^-} -\frac{1}{n}}}}2^{\frac{1}{p^-} -\frac{1}{n}+1}|A_{R_i}|^{\frac{1}{n}+\frac{1}{q_{A_R}^+}-\frac{1}{q_{A_R}^-}}(\log (e+\frac{1}{|A_{R_i}|}))^{r_{A_R}^+}\\
            &=& C_1 \frac{(1+||a||_{\infty})^{\frac{1}{p^-}}}{a_0^{\frac{1}{\frac{1}{q^-} -\frac{1}{n}}}} 2^{\frac{1}{p^-} -\frac{1}{n}+1}\frac{|A_R|^{\eta_R}}{2^{i \eta_R}}(\log (e+\frac{2^i}{|A_{R}|}))^{r_{A_R}^+},
\end{eqnarray*}
where $\eta_R :=\frac{1}{n}+\frac{1}{q_{A_R}^+}-\frac{1}{q_{A_R}^-}.$
Since we have $(\log (e+\frac{2^i}{|A_{R}|}))^{r_{A_R}^+}\leq i^{r_{A_R}^+}(\log (e+\frac{2}{|A_{R}|}))^{r_{A_R}^+}\leq 2^{r^+}i^{r_{A_R}^+}(\log (e+\frac{1}{|A_{R}|}))^{r_{A_R}^+}$ for $i\geq 1,$ we obtain 
\begin{eqnarray*}     
  R_i -R_{i+1} 
             &\leq & c_3\frac{i^{r_{A_R}^+}|A_R|^{\eta_R}}{2^{i \eta_R}}(\log (e+\frac{1}{|A_{R}|}))^{r_{A_R}^+},\\
            &\leq & c_3\frac{i^{r_{A_R}^+}|A_R|^{\eta_R}}{2^{i \eta_R}}(\log (e+\frac{1}{|A_{R}|}))^{r_{A_R}^+},
\end{eqnarray*}
where $c_3= C_1\frac{(1+||a||_{\infty})^{\frac{1}{p^-}}}{a_0^{\frac{1}{\frac{1}{q^-} -\frac{1}{n}}}}2^{\frac{1}{p^-} -\frac{1}{n}+1}2^{r^+}.$

Note that $\eta_R \geq \eta >0 ,$ and using the integral test we have $\sum_{i=1}^{\infty} {i^{r_{A_R}^+}}{2^{-i\eta_R}}\leq \frac{(r_{A_R}^+)!}{(\eta_R \log2)^{(r_{A_R}^+ +1)}}$ and hence
\begin{eqnarray*}     
 R = \sum_{i=0}^{\infty }(R_i -R_{i+1}) 
 &\leq& c_3 |A_R|^{\eta_R}(\log (e+\frac{1}{|A_{R}|}))^{r_{A_R}^+} \Big(1+\sum_{i=1}^{\infty} {i^{r_{A_R}^+}}{2^{-i\eta_R}}\Big)\\
  &\leq& c_3 |A_R|^{\eta_R}(\log (e+\frac{1}{|A_{R}|}))^{r_{A_R}^+}\Big(\frac{(r_{A_R}^+)!}{(\eta_R \log2)^{(r_{A_R}^+ +1)}}+1\Big)\\
   &\leq& c_3 |A_R|^{\eta_R}(\log (e+\frac{1}{|A_{R}|}))^{r_{A_R}^+}\Big(\frac{(r_{A_R}^+)!}{(\eta \log2)^{(r_{A_R}^+ +1)}}+1\Big)\\
  &\leq& c_3 |A_R|^{\eta_R}(\log (e+\frac{1}{|A_{R}|}))^{r_{A_R}^+}\Big(\frac{(r^+)!}{({\min(1,\eta)} \log2)^{(r^+ +1)}}+1\Big).
 \end{eqnarray*}
Moreover, since $ c_4:=1/ \max\Big\{1,\Big(\frac{c_3(r^+)!}{({\min(1,\eta)} \log2)^{(r^+ +1)}}+c_3\Big)\Big\} \leq 1,$ one has 
\begin{eqnarray}   
 |A_R| (\log (e+\frac{1}{|A_{R}|}))^\frac{r_{A_R}^+}{\eta_R}  \geq   c_4^{1 /\eta_R}  R^{1/\eta_R}
 \geq  c_4^{ 1 /\eta} R^{1 /\eta_R}
= c_4^{1 / \eta} R^n R^{\beta_R / \eta_R} ,
\end{eqnarray}
where $\beta_R := 1-n \eta_R$. Now we can proceed in a similar way as in Case-1 to obtain
$c
R^n (\log (e+\frac{1}{R}))^\frac{-r^+}{\eta} \leq c
 R^n (\log (e+\frac{1}{R}))^\frac{-r_{A_R}^+}{\eta_R} \leq | B_R(x) \cap \Omega | $. So, $\Omega$ satisfies the log-measure density condition.\qed
\section{Proof of Theorem \ref{main theorem 2}}
 \noindent We first prove three lemmas to estimate the norm of the characteristic function of a measurable set. The proofs are similar to the lemmas we proved in the previous section and hence we will provide only the main steps here.
 
 \begin{lemma}\label{Norm 1 lemma}
Let $\Phi:\Omega\times[0,\infty)\rightarrow[0,\infty]$  be given by $\Phi(x,t):=t^{p(x)}+a(x)t^{p(x)}(\log(e+t))^{r(x)}$ with $r(x)\geq 0$ for all $x$ and let $A\subset\Omega$ be a measurable set. Then 
\begin{equation}\label{Norm 1 lemma proof equation}
     \min \{|A|^\frac{1}{p_{A}^+},|A|^\frac{1}{p_{A}^-} \} \leq \|1_{A}\|_{L^{\Phi(\cdot,\cdot)}(\Omega)} \leq 2(1+||a||_{\infty})^{\frac{1}{p_{A}^-}}\max \{ |A|^\frac{1}{p_{A}^+} \log^{r_{A}^+} (e+\frac{1}{|A|}),|A|^\frac{1}{p_{A}^-} \log^{r_{A}^+}(1+e) \}
\end{equation}  
\end{lemma}
\begin{proof}
The second inequality of \eqref{Norm 1 lemma proof equation} follows by an argument analogous to that used to establish the second inequality of \eqref{Norm 4 lemma proof equation} by replacing $q_A^+$ with $ p_A^+$, and $q_A^-$ with $ p_A^-$ respectively.\\
Let us now prove the first inequality of \eqref{Norm 1 lemma proof equation}. Let $u<|A|$ and assume first that $u\leq 1$. Then 
\begin{equation*}
 \int_A \Phi\Big(x,\frac{1}{u^{\frac{1}{p_{A}^-}}}\Big)dx = \int_A \frac{1}{u^{\frac{p(x)}{p_{A}^-}}}+\frac{a(x)\Big(\log\big(e+\frac{1}{u^{\frac{1}{p_{A}^-}}}\big)\Big)^{r(x)}}{u^{\frac{p(x)}{p_{A}^-}}} dx \geq \frac{|A|}{u} >1.
\end{equation*}
 Hence, we get $u^\frac{1}{p_{A}^-}  \leq \|1_{A}\|_{L^{\Phi(\cdot,\cdot)}(\Omega)}$. If $u>1$, we can similarly show that $u^\frac{1}{p_{A}^+}  \leq \|1_{A}\|_{L^{\Phi(\cdot,\cdot)}(\Omega)}$. The first inequality follows as $u \rightarrow |A|^-$.
\end{proof}
\begin{lemma}\label{Norm 2 lemma}
Let $\Phi:\Omega\times[0,\infty)\rightarrow[0,\infty]$  be given by $\Phi(x,t):=t^{p(x)}+a(x)t^{p(x)}(\log(e+t))^{r(x)}$ with $r(x) < 0$ for all $x$ and let $A\subset \Omega$ be a measurable set. Then 
\begin{equation}\label{Norm 2}
     \min \{|A|^\frac{1}{p_{A}^+},|A|^\frac{1}{p_{A}^-} \} \leq \|1_{A}\|_{L^{\Phi(\cdot,\cdot)}(\Omega)} \leq  \max \{2(1+||a||_{\infty})^{\frac{1}{p_{A}^-}}|A|^\frac{1}{p_{A}^+},2(1+||a||_{\infty})^{\frac{1}{p_{A}^-}}|A|^\frac{1}{p_{A}^-} \}.
\end{equation}
\end{lemma}
\begin{proof}
The second inequality of \eqref{Norm 2} follows similarly to the proof of the second inequality of \eqref{Norm 5} by replacing $q_A^+$ with $ p_A^+$, and $q_A^-$ with $ p_A^-$ respectively.\\
Let us now prove the first inequality of \eqref{Norm 2}. Let $u<|A|$, and assume first that $u\leq 1$. Then 
\begin{equation*}
 \int_A \Phi\Big(x,\frac{1}{u^{\frac{1}{p_{A}^-}}}\Big)dx = \int_A \frac{1+a(x)\Big(\log\big(e+\frac{1}{u^{\frac{1}{p_{A}^-}}}\big)\Big)^{r(x)}}{u^{\frac{p(x)}{p_{A}^-}}} dx
 \geq \frac{|A|}{u} >1.
\end{equation*}
 Hence, we get $u^\frac{1}{p_{A}^-}  \leq \|1_{A}\|_{L^{\Phi(\cdot,\cdot)}(\Omega)}$. If $u>1$, we can similarly show that $u^\frac{1}{p_{A}^+}  \leq \|1_{A}\|_{L^{\Phi(\cdot,\cdot)}(\Omega)}$. The first inequality follows as $u \rightarrow |A|^-$.
\end{proof}
\begin{lemma}\label{Norm 3 lemma}
Let $\Phi:\Omega\times[0,\infty)\rightarrow[0,\infty)$  be given by $\Phi(x,t):=t^{p(x)}+a(x)t^{p(x)}(\log(e+t))^{r(x)}$ where $r(x) < 0$ for some $x $ and $r(x)\geq 0$ for some $x$. Assume $A\subset \Omega$ is a measurable set with $|A|< \frac{1}{2}$. Then there exists constant $b_2>0$ such that
\begin{equation}\label{Norm 3}
     |A|^\frac{1}{p_{A}^-}  \leq \|1_{A}\|_{L^{\Phi(\cdot,\cdot)}(\Omega)} \leq b_2 \max  4(1+||a||_{\infty})^{\frac{1}{p_{A}^-}} \{ |A|^\frac{1}{p_{A}^+} \big(\log (e+\frac{1}{|A|})\big)^{r_{A}^+} ,|A|^\frac{1}{p_{A}^-} \big(\log(1+e)\big)^{r_{A}^+} \}.
\end{equation}
\end{lemma}
\begin{proof}
The second inequality of \eqref{Norm 3} follows similarly to the proof of the second inequality of \eqref{Norm 6} by replacing $q_A^+$ with $ p_A^+$, and $q_A^-$ with $ p_A^- $ respectively.\\
Let us now prove the first inequality of \eqref{Norm 3}. Let $u<|A|$. Then 
\begin{eqnarray*}
\int_A \Phi\bigg(x,\frac{1}{u^{\frac{1}{p_{A}^-}}}\bigg)dx 
&=&  \int_A \frac{1}{u^{\frac{p(x)}{p_{A}^-}}}+\frac{a(x)\Big(\log\big(e+\frac{1}{u^{\frac{1}{p_{A}^-}}}\big)\Big)^{r(x)}}{u^{\frac{p(x)}{p_{A}^-}}} dx  \\
&=& \int_{A \cap (x: r(x) \geq 0)} \frac{1}{u^{\frac{p(x)}{p_{A}^-}}}+\frac{a(x)\Big(\log\big(e+\frac{1}{u^{\frac{1}{p_{A}^-}}}\big)\Big)^{r(x)}}{u^{\frac{p(x)}{p_{A}^-}}} dx +\\
& &\int_{A \cap (x: r(x) < 0)} \frac{1}{u^{\frac{p(x)}{p_{A}^-}}}+\frac{a(x)\Big(\log\big(e+\frac{1}{u^{\frac{1}{p_{A}^-}}}\big)\Big)^{r(x)}}{u^{\frac{p(x)}{p_{A}^-}}} dx  \\
&\geq& \frac{|A|}{u} > 1.
\end{eqnarray*}
Hence we get  $u^\frac{1}{p_{A}^-}   \leq \|1_{A}\|_{L^{\Phi(\cdot,\cdot)}(\Omega)}.$ The first inequality follows as $u \rightarrow |A|^-$.
\end{proof}

 Now, to prove Theorem \ref{main theorem 2}, we need the following lemma.
  \begin{lemma}
 If we have the same assumptions as in Theorem \ref{main theorem 2}, then there exist positive constants $C_1$, $C_2$, $C_3$ such that for all $x$ in $\bar{\Omega}$ and every $R$ in $]0,1]$ we have
\begin{equation}\label{ineuqality 10}
    R- \tilde{R} \leq C_1(1+||a||_{\infty})^{\frac{1}{p^-}} 2^{\frac{1}{p^-} -\frac{1}{n}+1}\ |A_R|^{\frac{1}{p_{A_R}^+} -\frac{1}{p_{A_R}^-} +\frac{1}{n}} (\log (e+\frac{1}{|A_R|}))^{r_{A_R}^+} 
 \end{equation}
when $r(x) \geq 0$ for all $x$;
\begin{equation}\label{inequality 11}
    R- \tilde{R} \leq C_2 (1+||a||_{\infty})^{\frac{1}{p^-}}2^{\frac{1}{p^-} -\frac{1}{n}+1}\ |A_R|^{\frac{1}{p_{A_R}^+} -\frac{1}{p_{A_R}^-} +\frac{1}{n}} 
\end{equation}
when $r(x) < 0$ for all $x$; and 
\begin{equation}\label{inequlity 12}
    R- \tilde{R} \leq  C_3(1+||a||_{\infty})^{\frac{1}{p^-}} 2^{\frac{1}{p^-} -\frac{1}{n}}\ |A_R|^{\frac{1}{p_{A_R}^+} -\frac{1}{p_{A_R}^-} +\frac{1}{n}} (\log (e+\frac{1}{|A_R|}))^{r_{A_R}^+} 
\end{equation}
when  $r(x) < 0$ for some $x$ and $r(x)\geq 0$ for some $x$. 
\end{lemma}       
 \begin{proof}    
Proceeding as in the proof of Lemma \ref{lemma for theorem 1.3}, we obtain the following
\begin{eqnarray*}
 R- \tilde{R} \leq {c_2} \frac{\|1_{B_R}\|_{L^{\Phi(\cdot,\cdot)(\Omega)}}}{ \| 1_{B_{\tilde{R}}}\|_{L^{\Psi(\cdot,\cdot)(\Omega)}}},
\end{eqnarray*}
where $c_2:=2c_1\max\{1,\tilde{c}\}.$\\ 

Case-1 ($r(x) \geq 0 $ for all $x\in \Omega$): Using the norm estimates for $\Phi$ in Lemma \ref{Norm 1 lemma}, and the similar estimates for $\Psi,$ we get
 \begin{eqnarray*}
 R- \tilde{R} &\leq& c_2\ 2(1+||a||_{\infty})^{\frac{1}{p_{A_R}^-}}\frac{|A_R|^\frac{1}{p_{A_R}^+} (\log (e+\frac{1}{|A_R|}))^{r_{A_R}^+}} {|A_{\tilde{R}}|^\frac{1}{p*_{A_R^-}}} \\
              &=& c_2\ 2(1+||a||_{\infty})^{\frac{1}{p_{A_R}^-}}\frac{|A_R|^\frac{1}{p_{A_R}^+} (\log (e+\frac{1}{|A_R|}))^{r_{A_R}^+}} {|A_{\tilde R}|^{\frac{1}{p_{A_R}^-} -\frac{1}{n}}}\\
              &=&  c_2\ (1+||a||_{\infty})^{\frac{1}{p_{A_R}^-}}2^{\frac{1}{p_{A_R}^-} -\frac{1}{n}+1}|A_R|^{\frac{1}{p_{A_R}^+} -\frac{1}{p_{A_R}^-} +\frac{1}{n}} (\log (e+\frac{1}{|A_R|}))^{r_{A_R}^+}\\
              &\leq & c_2(1+||a||_{\infty})^{\frac{1}{p^-}} 2^{\frac{1}{p^-} -\frac{1}{n}+1}\ |A_R|^{\frac{1}{p_{A_R}^+} -\frac{1}{p_{A_R}^-} +\frac{1}{n}} (\log (e+\frac{1}{|A_R|}))^{r_{A_R}^+}, 
\end{eqnarray*}
as claimed.\\
Case-2 ($r(x) \leq 0$ for all $x\in \Omega$): Using the norm estimates for $\Phi$ in Lemma \ref{Norm 2 lemma} and the similar estimates for $\Psi,$ we get 
\begin{eqnarray*}
 R- \tilde{R} &\leq& c_2\ 2(1+||a||_{\infty})^{\frac{1}{p_{A_R}^-}}\frac{|A_R|^\frac{1}{p_{A_R}^+}} {|A_{\tilde{R}}|^\frac{1}{p*_{A_R}^-} } \\
              &=& c_2\ 2(1+||a||_{\infty})^{\frac{1}{p_{A_R}^-}}\frac{|A_R|^\frac{1}{p_{A_R}^+}} {|A_{\tilde R}|^{\frac{1}{p_{A_R}^-} -\frac{1}{n}}}\\
              &\equiv &  c_2 (1+||a||_{\infty})^{\frac{1}{p_{A_R}^-}}2^{\frac{1}{p_{A_R}^-} -\frac{1}{n}+1}|A_R|^{\frac{1}{p_{A_R}^+} -\frac{1}{p_{A_R}^-} +\frac{1}{n}}\\
              &\leq & c_2 (1+||a||_{\infty})^{\frac{1}{p^-}}2^{\frac{1}{p^-} -\frac{1}{n}+1}\ |A_R|^{\frac{1}{p_{A_R}^+} -\frac{1}{p_{A_R}^-} +\frac{1}{n}}.
\end{eqnarray*}
Case-3 ($r(x) < 0$ for some $x\in \Omega$ and $r(x)\geq 0$ for some $x\in \Omega$): Using the norm estimates for $\Phi$ in Lemma \ref{Norm 3 lemma} and the similar estimates for $\Psi,$ we get 
\begin{eqnarray*}
 R- \tilde{R} &\leq& c_2\ 4(1+||a||_{\infty})^{\frac{1}{p_{A_R}^-}}\frac{|A_R|^\frac{1}{p_{A_R}^+} (\log (e+\frac{1}{|A_R|}))^{r_{A_R}^+}} {|A_{\tilde{R}}|^\frac{1}{p*_{A_R^-}}} \\
              &=& c_2\ 4(1+||a||_{\infty})^{\frac{1}{p_{A_R}^-}}\frac{|A_R|^\frac{1}{p_{A_R}^+} (\log (e+\frac{1}{|A_R|}))^{r_{A_R}^+}} {|A_{\tilde R}|^{\frac{1}{p_{A_R}^-} -\frac{1}{n}}}\\
              &=&  c_3\ (1+||a||_{\infty})^{\frac{1}{p_{A_R}^-}}2^{\frac{1}{p_{A_R}^-} -\frac{1}{n}}|A_R|^{\frac{1}{p_{A_R}^+} -\frac{1}{p_{A_R}^-} +\frac{1}{n}} (\log (e+\frac{1}{|A_R|}))^{r_{A_R}^+}\\
              &\leq & c_3(1+||a||_{\infty})^{\frac{1}{p^-}} 2^{\frac{1}{p^-} -\frac{1}{n}}\ |A_R|^{\frac{1}{p_{A_R}^+} -\frac{1}{p_{A_R}^-} +\frac{1}{n}} (\log (e+\frac{1}{|A_R|}))^{r_{A_R}^+},
\end{eqnarray*}
where $c_3= 4c_2$.
\end{proof}
\textbf{Proof of Theorem \ref{main theorem 2}}
Construct the sequence $\{  R_i \}$ by setting $R_0:=R$, and then define $R_{i+1}:=\tilde{R_i}$ inductively for $i \geq 0$.
It follows that
\[
     |A_{R_{i}}|=\frac{1}{2^i}\ | A_R|, 
\]
with $\lim_{i \rightarrow \infty} R_i=0$.\\

Case-1 ($r(x) \geq 0 $ for all $x\in \Omega$): Using the inequality \eqref{ineuqality 10} one obtains
\begin{eqnarray*}     
  R_i -R_{i+1} &\leq& C_1(1+||a||_{\infty})^{\frac{1}{p^-}} 2^{\frac{1}{p^-} -\frac{1}{n}+1} |A_{R_i}|^{\frac{1}{n}+\frac{1}{p_{A_{R_i}}^+}-\frac{1}{p_{A_{R_i}}^-}}(\log (e+\frac{1}{|A_{R_i}|}))^{r_{A_{R_i}}^+}\\ 
	      &\leq&  C_1 (1+||a||_{\infty})^{\frac{1}{p^-}} 2^{\frac{1}{p^-} -\frac{1}{n}+1}|A_{R_i}|^{\frac{1}{n}+\frac{1}{p_{A_R}^+}-\frac{1}{p_{A_R}^-}}(\log (e+\frac{1}{|A_{R_i}|}))^{r_{A_R}^+}\\
             &=& C_1 (1+||a||_{\infty})^{\frac{1}{p^-}} 2^{\frac{1}{p^-} -\frac{1}{n}+1}\frac{|A_R|^{\eta_R}}{2^{i \eta_R}}(\log (e+\frac{2^i}{|A_{R}|}))^{r_{A_R}^+},
\end{eqnarray*}
where $\eta_R :=\frac{1}{n}+\frac{1}{p_{A_R}^+}-\frac{1}{p_{A_R}^-}.$  
Since we have $(\log (e+\frac{2^i}{|A_{R}|}))^{r_{A_R}^+}\leq i^{r_{A_R}^+}(\log (e+\frac{2}{|A_{R}|}))^{r_{A_R}^+}\leq 2^{r^+}i^{r_{A_R}^+}(\log (e+\frac{1}{|A_{R}|}))^{r_{A_R}^+}$ for $i\geq 1,$ we obtain 
\begin{eqnarray*}     
  R_i -R_{i+1} 
             &\leq & c_3\frac{i^{r_{A_R}^+}|A_R|^{\eta_R}}{2^{i \eta_R}}(\log (e+\frac{1}{|A_{R}|}))^{r_{A_R}^+},
\end{eqnarray*}
where $c_3={C_1}(1+||a||_{\infty})^{\frac{1}{p^-}} 2^{\frac{1}{p^-} -\frac{1}{n}+1}2^{r^+}.$\\
Applying the same argument to $p$ as in the proof of Theorem \(\ref{main theorem 3}\), we obtain a constant \(\eta>0\), independent of both \(x\) and \(R\), such that
$\frac{1}{n}+\frac{1}{p_{A_R}^{+}}-\frac{1}{p_{A_R}^{-}}
=: \eta_R\geq \eta>0$
for every \(R\leq r_0\).\\
Note that $\eta_R \geq \eta >0$  and using the integral test we have $\sum_{i=1}^{\infty} {i^{r_{A_R}^+}}{2^{-i\eta_R}}\leq \frac{(r_{A_R}^+)!}{(\eta_R \log2)^{(r_{A_R}^+ +1)}},$ where for any positive real number $a$ we have $a! = \Gamma(a+1),$ and hence
\begin{eqnarray*}     
 R = \sum_{i=0}^{\infty }(R_i -R_{i+1}) 
 &\leq& c_3 |A_R|^{\eta_R}(\log (e+\frac{1}{|A_{R}|}))^{r_{A_R}^+} \Big(1+\sum_{i=1}^{\infty} {i^{r_{A_R}^+}}{2^{-i\eta_R}}\Big)\\
  &\leq& c_3 |A_R|^{\eta_R}(\log (e+\frac{1}{|A_{R}|}))^{r_{A_R}^+}\Big(\frac{(r_{A_R}^+)!}{(\eta_R \log2)^{(r_{A_R}^+ +1)}}+1\Big)\\
   &\leq& c_3 |A_R|^{\eta_R}(\log (e+\frac{1}{|A_{R}|}))^{r_{A_R}^+}\Big(\frac{(r_{A_R}^+)!}{(\eta \log2)^{(r_{A_R}^+ +1)}}+1\Big)\\
   &\leq& c_3 |A_R|^{\eta_R}(\log (e+\frac{1}{|A_{R}|}))^{r_{A_R}^+}\Big(\frac{(r^+)!}{({\min(1,\eta)} \log2)^{(r^+ +1)}}+1\Big).
 \end{eqnarray*}
Moreover, since $ c_4:=1/ \max\Big\{1,\Big(\frac{c_3(r^+)!}{({\min(1,\eta)} \log2)^{(r^+ +1)}}+c_3\Big)\Big\} \leq 1,$ one has 
\begin{eqnarray}    \label{AR}
 |A_R| (\log (e+\frac{1}{|A_{R}|}))^\frac{r_{A_R}^+}{\eta_R}  \geq   c_4^{1 /\eta_R}  R^{1/\eta_R}
 \geq  c_4^{ 1 /\eta} R^{1 /\eta_R}
 = c_4^{1 / \eta} R^n R^{\beta_R / \eta_R} ,
\end{eqnarray}
where $\beta_R := 1-n \eta_R$.

Therefore, from equation \eqref{AR} one sees that if a positive lower bound for $R^{\beta_R / \eta_R}$ 
is provided, the proof of Theorem \ref{main theorem 2} will be completed. Applying the same argument to $p$ as in the proof of Theorem \(\ref{main theorem 3}\) we get
\[  
R^{\frac{\beta_R}{\eta_R}} \geq c_5 >0.
\]
 Taking $f(t)=t (\log (e+\frac{1}{t}))^\frac{q_{A_R}^+}{\eta_R}$, we see that \eqref{AR} becomes $f(|A_R|)\geq cR^n,$ where $c := c_4^{1 / \eta} c_5$ and hence $|A_R|\geq f^{-1}(cR^n)$ which further implies that $c
 R^n (\log (e+\frac{1}{R}))^\frac{-r^+}{\eta} \leq c
 R^n (\log (e+\frac{1}{R}))^\frac{-r_{A_R}^+}{\eta_R} \leq | B_R(x) \cap \Omega | $. So, $\Omega$ satisfies the log-measure density condition.\\

Case-2 ($r(x) \leq 0$ for all $x\in \Omega$): Using the inequality \eqref{inequality 11} one obtains 
\begin{eqnarray*}
R_i -R_{i+1} &\leq& C_2 (1+||a||_{\infty})^{\frac{1}{p^-}}2^{\frac{1}{p^-} -\frac{1}{n}+1}|A_{R_i}|^{\frac{1}{n}+\frac{1}{p_{A_{R_i}}^+}-\frac{1}{p_{A_{R_i}}^-}}\\ 
	     &\leq&  C_2 (1+||a||_{\infty})^{\frac{1}{p^-}}2^{\frac{1}{p^-} -\frac{1}{n}+1}|A_{R_i}|^{\frac{1}{n}+\frac{1}{p_{A_R}^+}-\frac{1}{p_{A_R}^-}}\\
             &=& C_2 (1+||a||_{\infty})^{\frac{1}{p^-}}2^{\frac{1}{p^-} -\frac{1}{n}+1}\frac{|A_R|^{\eta_R}}{2^{i \eta_R}} \\
             &\leq & c_6{|A_R|^{\eta_R}}{2^{-i \eta_R}},
\end{eqnarray*}
where $c_6=C_2(1+||a||_{\infty})^{\frac{1}{p^-}}2^{\frac{1}{p^-} -\frac{1}{n}+1} .$ Note that $\eta_R \geq \eta>0 $ and hence
\begin{eqnarray*}     
 R = \sum_{i=0}^{\infty }(R_i -R_{i+1}) 
 &\leq& c_6 |A_R|^{\eta_R} \Big(\sum_{i=0}^{\infty} {2^{-i\eta_R}}\Big)\\
 &=& c_6 |A_R|^{\eta_R} \frac{1}{1-2^{-\eta_R}}\\
 &\leq&c_6 |A_R|^{\eta_R} \frac{1}{1-2^{-\eta}}\\
 &=& c_7 |A_R|^{\eta_R},
 \end{eqnarray*}
 where $c_7=  \frac{c_6}{1-2^{-\eta}}$.
 Moreover, since $ c_8:=1/ \max\{1,c_7\} \leq 1$ one has
 \begin{eqnarray*}   
 |A_R|   \geq   c_8^{1 /\eta_R}  R^{1/\eta_R}
 \geq  c_8^{ 1 /\eta} R^{1 /\eta_R}
 = c_8^{1 / \eta} R^n R^{\beta_R / \eta_R} ,
\end{eqnarray*}
where $\beta_R := 1-n \eta_R$.
Now we can proceed in a similar way as in Case-1 to obtain $cR^n  \leq  | B_R(x) \cap \Omega | $. So, $\Omega$ satisfies the measure density condition.\\
 
  Case-3 ($r(x) < 0$ for some $x\in \Omega$ and $r(x)\geq 0$ for some $x\in \Omega$): Using the inequality \eqref{inequlity 12}, one obtains 
  \begin{eqnarray*}     
  R_i -R_{i+1} &\leq& C_1(1+||a||_{\infty})^{\frac{1}{p^-}} 2^{\frac{1}{p^-} -\frac{1}{n}+1} |A_{R_i}|^{\frac{1}{n}+\frac{1}{p_{A_{R_i}}^+}-\frac{1}{p_{A_{R_i}}^-}}(\log (e+\frac{1}{|A_{R_i}|}))^{r_{A_{R_i}}^+}\\ 
	      &\leq&  C_1 (1+||a||_{\infty})^{\frac{1}{p^-}} 2^{\frac{1}{p^-} -\frac{1}{n}+1}|A_{R_i}|^{\frac{1}{n}+\frac{1}{p_{A_R}^+}-\frac{1}{p_{A_R}^-}}(\log (e+\frac{1}{|A_{R_i}|}))^{r_{A_R}^+}\\
            &=& C_1 (1+||a||_{\infty})^{\frac{1}{p^-}} 2^{\frac{1}{p^-} -\frac{1}{n}+1}\frac{|A_R|^{\eta_R}}{2^{i \eta_R}}(\log (e+\frac{2^i}{|A_{R}|}))^{r_{A_R}^+},
\end{eqnarray*}
Since we have $(\log (e+\frac{2^i}{|A_{R}|}))^{r_{A_R}^+}\leq i^{r_{A_R}^+}(\log (e+\frac{2}{|A_{R}|}))^{r_{A_R}^+}\leq 2^{r^+}i^{r_{A_R}^+}(\log (e+\frac{1}{|A_{R}|}))^{r_{A_R}^+}$ for $i\geq 1,$ we obtain 
\begin{eqnarray*}     
  R_i -R_{i+1} 
             &\leq & c_3\frac{i^{r_{A_R}^+}|A_R|^{\eta_R}}{2^{i \eta_R}}(\log (e+\frac{1}{|A_{R}|}))^{r_{A_R}^+},\\
            &\leq & c_3\frac{i^{r_{A_R}^+}|A_R|^{\eta_R}}{2^{i \eta_R}}(\log (e+\frac{1}{|A_{R}|}))^{r_{A_R}^+},
\end{eqnarray*}
where $c_3={C_1}(1+||a||_{\infty})^{\frac{1}{p^-}} 2^{\frac{1}{p^-} -\frac{1}{n}+1}2^{r^+}.$

Note that $\eta_R \geq \eta >0,$ and using the integral test we have $\sum_{i=1}^{\infty} {i^{r_{A_R}^+}}{2^{-i\eta_R}}\leq \frac{(r_{A_R}^+)!}{(\eta_R \log2)^{(r_{A_R}^+ +1)}}$ and hence
\begin{eqnarray*}     
 R = \sum_{i=0}^{\infty }(R_i -R_{i+1}) 
 &\leq& c_3 |A_R|^{\eta_R}(\log (e+\frac{1}{|A_{R}|}))^{r_{A_R}^+} \Big(1+\sum_{i=1}^{\infty} {i^{r_{A_R}^+}}{2^{-i\eta_R}}\Big)\\
  &\leq& c_3 |A_R|^{\eta_R}(\log (e+\frac{1}{|A_{R}|}))^{r_{A_R}^+}\Big(\frac{(r_{A_R}^+)!}{(\eta_R \log2)^{(r_{A_R}^+ +1)}}+1\Big)\\
   &\leq& c_3 |A_R|^{\eta_R}(\log (e+\frac{1}{|A_{R}|}))^{r_{A_R}^+}\Big(\frac{(r_{A_R}^+)!}{(\eta \log2)^{(r_{A_R}^+ +1)}}+1\Big)\\
  &\leq& c_3 |A_R|^{\eta_R}(\log (e+\frac{1}{|A_{R}|}))^{r_{A_R}^+}\Big(\frac{(r^+)!}{({\min(1,\eta)} \log2)^{(r^+ +1)}}+1\Big).
 \end{eqnarray*}
Moreover, since $ c_4:=1/ \max\Big\{1,\Big(\frac{c_3(r^+)!}{({\min(1,\eta)} \log2)^{(r^+ +1)}}+c_3\Big)\Big\} \leq 1,$ one has 
\begin{eqnarray*}   
 |A_R| (\log (e+\frac{1}{|A_{R}|}))^\frac{r_{A_R}^+}{\eta_R}  \geq   c_4^{1 /\eta_R}  R^{1/\eta_R}
 \geq  c_4^{ 1 /\eta} R^{1 /\eta_R}
= c_4^{1 / \eta} R^n R^{\beta_R / \eta_R} ,
\end{eqnarray*}
where $\beta_R := 1-n \eta_R$. Now we can proceed in a similar way as in Case-1 to obtain
$c
R^n (\log (e+\frac{1}{R}))^\frac{-r^+}{\eta} \leq c
 R^n (\log (e+\frac{1}{R}))^\frac{-r_{A_R}^+}{\eta_R} \leq | B_R(x) \cap \Omega | $. So, $\Omega$ satisfies the log-measure density condition.\qed

\vspace{.5 cm}

\bigskip
\noindent\textbf{Conflict of interest:} The authors declare that they have no conflict of interest.\\
\textbf{Data Availability Statement:} Data sharing is not applicable to this article, as no data sets were
generated or analyzed during the current study.\\
\noindent\textbf{Contributions:} All authors have contributed equally to the manuscript.

\def\bibname{References}
\bibliography{double}
\bibliographystyle{acm}

\bigskip

\noindent{\small Ankur Pandey}\\
\small{Department of Mathematics,}\\
\small{Birla Institute of Technology and Science-Pilani, Hyderabad Campus,}\\
\small{Hyderabad-500078, India} \\
{\tt p20210424@hyderabad.bits-pilani.ac.in; pandeyankur600@gmail.com}\\
\\
\noindent{\small Nijjwal Karak}\\
\small{Department of Mathematics,}\\
\small{Birla Institute of Technology and Science-Pilani, Hyderabad Campus,}\\
\small{Hyderabad-500078, India} \\
{\tt nijjwal@gmail.com; nijjwal@hyderabad.bits-pilani.ac.in}\\

\end{document}